\documentclass[smallextended,final,numbook,envcountsame]{svjour3}    
\smartqed
\usepackage{amsmath,amstext, amssymb,color,mathptmx}
\journalname{Mathematische Annalen}
\usepackage[colorlinks, linkcolor=reference,
            citecolor=citation,
          urlcolor=e-mail, linktoc=all]{hyperref}

\newcommand{\protectedtexorpdfstring}[2]{%
  \texorpdfstring{\unexpanded{\unexpanded{#1}}}{#2}%
}

\newcommand{\txymatrix}[2]{%
  \protectedtexorpdfstring
    {\catcode`\@=12 \scantokens{\xymatrix{#1}}}
    {#2}%
}

\usepackage[paperwidth=6.4in,paperheight=9.2in,top=0.9in,left=0.8in,right=0.8in,bottom=0.9in]{geometry}

\definecolor{e-mail}{rgb}{0,.40,.80}
\definecolor{reference}{rgb}{.20,.60,.22}
\definecolor{citation}{rgb}{0,.40,.80}
\usepackage[english]{babel}
\usepackage[all]{xy}

\usepackage{color}
\def\beq{\begin{equation}}
\def\eeq{\end{equation}}

\def\Z{{\mathbb Z}}
\def\Q{{\mathbb Q}}

\def\C{{\mathbb C}}

\def\D{{\mathbb D}}
\def\I{{\mathbb I}}

\def\K{{\mathbf k}}
\def\k{{\mathbf k}}

\def\o{\omega}

\def\p{\prime}

\def\cM{{\mathcal M}}
\def\cN{{\mathcal N}}

\def\cL{{\mathcal{ L}}}

\def\cR{{\mathcal R}}

\def\cU{{\mathcal U}}
\def\cV{{\mathcal V}}
\def\cW{{\mathcal W}}

\def\cX{{\mathcal X }}

\def\dt{\delta}
\def\dx{\partial}

\def\ga{\mathfrak{g}}
\def\gl{\mathfrak{gl}}
\def\ha{\mathfrak{h}}

\def\wtilde{\widetilde}

\def\D{\Delta}

\def\p{\partial}

\def\Hom{\operatorname{Hom}}
\def\Gl{\operatorname{GL}}
\def\GL{\operatorname{GL}}

\def\Ad{\operatorname{Ad}}
\def\Sl{\operatorname{SL}}
\def\Lie{\operatorname{Lie}}
\def\Char{\operatorname{char}}
\def\SL{\operatorname{SL}}
\def\PSL{\operatorname{PSL}}

\def\Stab{\operatorname{Stab}}
\def\Gal{\operatorname{Gal}}
\def\Galdelta{\operatorname{Gal}^\delta}
\def\Span{\operatorname{span}}

\def\Diff{\operatorname{Diff}}
\def\id{\operatorname{id}}

\def\ord{\operatorname{ord}}

\def\Stabdelta{\operatorname{Stab}^\delta}

\def\Ga{\bold{G}_a}

\def\Vect{\operatorname{Vect}}
\def\de{\delta}

\begin{document}
\title{Calculating differential {G}alois groups of parametrized differential equations, with applications to hypertranscendence\thanks{This work was partially supported by 
by ANR-11-LABX-0040-CIMI within the program ANR-11-IDEX-0002-02, by ANR-10-JCJC 0105,  by the NSF grants CCF-0952591 and DMS-1413859, by the NSA grant H98230-15-1-0245,  by the ISF grant 756/12, and by the Minerva foundation with funding from the Federal German Ministry for Education and Research.}}
\titlerunning{{G}alois groups of parametrized differential equations, with applications to hypertranscendence}
\author{Charlotte Hardouin \and  Andrei Minchenko \and Alexey Ovchinnikov}

\institute{C. Hardouin \at
Institut de Math\'{e}matiques de Toulouse, 118 route de Narbonne,
31062 Toulouse Cedex 9, France\\ \email{hardouin@math.ups-tlse.fr}
\and
A. Minchenko\at 
The Weizmann Institute of Science, Faculty of Mathematics and Computer Science, 234 Herzl Street, Rehovot, 7610001, Israel\\\email{an.minchenko@gmail.com}
\and
A. Ovchinnikov \at 
CUNY Queens College, Department of Mathematics,
65-30 Kissena Blvd, Queens, NY 11367, USA and
CUNY Graduate Center, Department of Mathematics, 365 Fifth Avenue,
New York, NY 10016, USA\\ \email{aovchinnikov@qc.cuny.edu}
}

\date{Received:  8 June 2015 / Accepted: 16 June 2016}


\maketitle

\begin{abstract}
The main motivation of our work is to create an efficient algorithm that decides hypertranscendence of solutions of linear differential 
equations, via the parameterized differential and 
Galois theories. To achieve this, we expand the representation theory of linear differential algebraic groups and develop new algorithms that calculate unipotent radicals of parameterized differential Galois groups for differential equations whose coefficients are rational functions. P. Berman and M.F. Singer presented an algorithm calculating the differential Galois group for differential equations without parameters whose differential operator is a composition of two completely reducible differential operators. We use their algorithm as a part of our algorithm.   
As a result, we find an effective criterion for the algebraic independence of the solutions of parameterized differential equations and all of their derivatives
with respect to the parameter.
\keywords{Differential Galois theory \and Differential transcendence \and Special functions}
\end{abstract}

\tableofcontents

\section{Introduction}

A special function is said to be hypertranscendental if it does not satisfy any algebraic differential equation. 
The study  of functional hypertranscendence has recently appeared in various areas of mathematics.
In combinatorics, the question of the hypertranscendence of generating series is frequent because it gives information on the growth 
of the coefficients: for instance, the work of  Kurkova and Raschel \cite{KurkRaschel} solved a   famous conjecture about the differential algebraic behaviour of generating series of walks 
on the plane. Dreyfus, Roques, and Hardouin \cite{DHR} gave criteria to test the hypertranscendence 
of generating series associated to $p$-automatic sequences and more generally Mahler functions, generalizing the work of Nguyen \cite{NG}, Nishioka \cite{Nish}, and Rand\'{e} \cite{RandeThese}.  Also, when the derivation encodes the continuous
deformation of an auxiliary parameter, the hypertranscendence is connected to the notion of isomonodromic deformation (see the work of Mitschi and Singer \cite{MitschiSinger:MonodromyGroupsOfParameterizedLinearDifferentialEquationsWithRegularSingularities}).

  The work of  Cassidy, Hardouin, and Singer \cite{cassisinger,HardouinSinger} were motivated by a  study  of hypertranscendence using Galois theory. 
Starting from a linear functional equation  with coefficients in a field with a ``parametric'' derivation, they were able to construct a geometric object, called the parameterized differential Galois group,  whose symmetries control the algebraic relations between the solutions of the functional equation and all of their derivatives. The question of hypertranscendence of solutions of linear functional equations is thus reduced to the computation of the parameterized differential Galois groups of the equations (see for instance the work of Arreche \cite{Carlos2} on the incomplete gamma function $\gamma(x,t)$ and the work \cite{DHR}). The parameterized differential Galois groups are linear differential algebraic groups as introduced by Kolchin and   developed by Cassidy \cite{cassdiffgr}. These are groups of matrices whose entries 
satisfy systems of polynomial differential  equations, called  defining equations of the parameterized differential Galois group.

Then, in this context of Galois theory, one can address  a direct problem, that is, the question of the algorithmic computation of the parameterized differential Galois group. For linear functional equations of order $2$, one can find a Kovacic-type  algorithm  initiated by Dreyfus~\cite{Dreyfus}  and  completed by Arreche~\cite{CarlosISSAC}. In~\cite{MinOvSing}, Minchenko, Ovchinnikov, and Singer  gave an algorithm that allows to test 
if the parameterized differential Galois group is reductive and to compute the group in that case. In \cite{MinOvSingunip}, they also show how to compute
the parameterized differential Galois group if its quotient by the unipotent radical is conjugate to a group of matrices with constant entries with respect to the parametric derivations.  The algorithms of \cite{MinOvSingunip,MinOvSing} rely on bounds on the order of the defining equations of the parameterized differential
Galois group, which allows to use  the algorithm obtained by Hrushovski \cite{Hrushcomp} and has been further analyzed and improved by Feng \cite{RFdifferential2015} in the case of no parametric derivations.

In this paper, we study the parameterized differential Galois group of a differential operator of the form $L_1(L_2(y))=0$ where 
$L_1,L_2$ are completely reducible differential operators. This   situation  goes beyond the previously studied cases, because the parameterized Galois
group of such an equation is no longer reductive and its quotient by its unipotent radical might not be constant. If there is no parametric derivation, this problem was solved by Berman and Singer in \cite{BeSing} for differential operators and rephrased using Tannakian categories
by Hardouin \cite{HardouinSemCongres}. The general case is however more complicated because, unlike the case of no parameters, the 
order of the defining equations of the parameterized differential Galois group is no longer controlled by the order of the functional equation $L_1(L_2(y))=0$.
Therefore, we present an algorithm that relies on bounds  (see Section~\ref{sec:mergeconstantnonconstant}) and, in a generic situation, we find a  description of the parameterized differential Galois group. In this description, the  defining equations of the unipotent radical are obtained  by applying  standard operations to linear differential operators  (cf.~\cite{HardouinSemCongres}).    
  
  However, by a careful study of the extension of completely reducible representations of quasi-simple linear differential algebraic groups, we are able to deduce 
a complete and effective criterion to test the hypertranscendence of solutions of inhomogeneous linear differential equations (Theorem \ref{cor:criteria}).

  The paper is organized as follows. We start with a brief review of the basic notions in differential algebra, linear differential algebraic groups, and linear differential equations with parameters in Section~\ref{sec:notions}. {Our algorithmic results for calculating parameterized differential} Galois groups
are presented in Section~\ref{sec:main}.
Our effective criterion for hypertranscendence
of solutions of extensions of irreducible differential equations is contained in Section~\ref{sec:criterion}, which is preceded by Section~\ref{sec:repsplit}, where we extend  results of Minchenko and Ovchinnikov \cite{MinOvRepSL2} for the purposes of the hypertranscendence criterion. We use this criterion
 to prove hypertranscendence results for the Lommel differential equation in Section~\ref{sec:examples}.

\section{Preliminary notions}\label{sec:notions}
We shall start with some basic notions of differential algebra and then recall what linear differential algebraic groups  and their representations are.

\subsection{Differential algebra}
\begin{definition} A {\em differential ring} is a ring $R$ with a finite set $\Delta=\{\delta_1,\ldots,\delta_m\}$ of commuting derivations on $R$. A {\em $\Delta$-ideal} of $R$
is an ideal of $R$ stable under any derivation in $\Delta$. \end{definition}

 In the present paper, $\Delta$ will consist of one or two elements. Let $R$ be a $\Delta$-ring. For any $\delta \in \Delta$, we denote
 $$
 R^\delta= \{r \in R\:|\: \delta(r) = 0\},
 $$
 which is a $\Delta$-subring  of $R$ and is called the {\em ring of $\delta$-constants} of $R$. If $R$ is a field and a differential ring, then it is called a differential field, or $\Delta$-field for short. For example, $R = \Q(x,t)$, $\Delta=\{\delta,\partial\}$, and $\partial = \partial/\partial x$, $\delta=\partial/\partial t$, forms a differential field. The notion of $R$-$\Delta$-algebra is defined analogously.

The ring of $\Delta$-differential polynomials $K\{y_1,\ldots,y_n\}$ in the differential indeterminates, or $\Delta$-indeterminates,  $y_1,\ldots,y_n$ and with coefficients in a $\Delta$-field $(K,\Delta)$,  is the ring of polynomials in the indeterminates formally denoted $$\left\{\delta_1^{i_1}\cdot\ldots\cdot\delta_m^{i_m} y_i\:\big|\: i_1,\ldots,i_m\ge 0,\, 1\le i\le n\right\}$$ with coefficients in $K$. We endow this ring with a structure of $K$-$\Delta$-algebra by setting 
$$\delta_k \left(\delta_1^{i_1}\cdot\ldots\cdot \delta_m^{i_m} y_i \right)= \delta_1^{i_1} \cdot\ldots\cdot \delta_k^{i_k+1} \cdot\ldots\cdot \delta_m^{i_m} y_i.$$ 

\begin{definition}[{see \cite[Corollary~1.2(ii)]{Marker2000}}]
A differential field $(K,\Delta)$ is said to be differentially closed or $\Delta$-closed for short, if, for every (finite) set of $\Delta$-polynomials $F \subset K\{y_1,\ldots,y_n \}$, if the system of differential equations $F=0$ has a solution  with entries in some $\Delta$-field extension $L$, then it has a solution with entries in $K$.\end{definition}

 For $\partial \in \Delta$, the ring $K[\partial]$ of differential operators, or $\partial$-operators for short,  is the $K$-vector space with basis $1,\partial,\dots,\partial^n,\dots$ endowed  with the following multiplication rule: $$\partial\cdot a = a\cdot\partial + \partial(a).$$
To a $\partial$-operator $L$ as above, one can associate the linear homogeneous $\partial$-polynomial $$L(y)= a_n\partial^n y+\ldots+a_1\partial y + a_0 y \in K\{y\}.$$ In what follows, we  assume that every field is of characteristic zero.

\subsection{Linear differential algebraic groups  and their unipotent radicals}
In this section, we first introduce the basic terminology of Kolchin-closed sets, linear differential algebraic groups and their representations. We then define   unipotent radicals of linear differential algebraic groups, reductive linear differential algebraic groups  and their structural properties. We continue with the notion of  conjugation to constants of linear differential algebraic groups.

Let  $(\K,\delta)$ be a differentially closed field, $C=\K^\delta$, and  $(F,\delta)$  a $\delta$-subfield of $\k$.

\subsubsection{First definitions}

\begin{definition}\label{def:Kc} A {\it Kolchin-closed} (or $\delta$-closed, for short) set $W \subset \K^n$   is the set of common zeroes
of a system of $\delta$-polynomials with coefficients in $\K$, that is, there exists  $S \subset \K\{y_1,\dots,y_n\}$ such that
$$
W = \left\{ a \in \K^n\:|\: f(a) = 0 \mbox{ for all } f \in S \right\}.$$
We say that $W$ is defined over $F$ if $W$ is the set of zeroes of  $\delta$-polynomials with coefficients in  $F$. More generally, for an $F$-$\delta$-algebra R, $$
W(R) = \left\{ a \in R^n\:|\:   f(a) = 0 \mbox{ for all } f \in S \right\}.$$
\end{definition}

\begin{definition}
If $W \subset \K^n$ is a  Kolchin-closed set defined  over $F$,  the $\delta$-ideal  $$\I(W) = \{ f\in F\{y_1,  \ldots , y_n\} \ | \ f(w) = 0 \mbox{ for all } \ w\in W(\K)\}$$
is called the defining $\delta$-ideal of $W$ over $F$.
Conversely, for a  subset  $S$ of $F\{y_1,\dots,y_n\}$, the following subset  is $\delta$-closed in  $\K^n$ and defined over $F$:
$$
\bold{V}(S)=  \left\{ a \in \K^n\:|\: f(a)= 0 \mbox{ for all } f \in S \right\}.$$ 
\end{definition}
\begin{remark}
Since every  radical $\delta$-ideal of  $F\{y_1,  \ldots , y_n\}$ is generated as a radical $\delta$-ideal by a finite set of $\delta$-polynomials  (see, for example, \cite[Theorem, page~10]{RittDiffalg}, \cite[Sections~VII.27-28]{Kapldiffalg}),  the Kolchin topology is {\em Ritt--Noetherian}, that is, every strictly decreasing chain of Kolchin-closed sets has a finite length.
\end{remark}

\begin{definition}
Let $W \subset \K^n$ be a $\delta$-closed set defined over $F$. The $\delta$-coordinate ring $F\{W\}$ of 
$W$ over $F$ is the $F$-$\Delta$-algebra
$$
F\{W\} = F\{y_1,\ldots,y_n\}\big/\I(W).
$$
If $F\{W\}$ is an integral domain, then $W$ is said to be {\it irreducible}. This  is equivalent to $\I(W)$ being a prime $\delta$-ideal.
\end{definition}

\begin{example}
The affine space $\bold{A }^n$ is the irreducible Kolchin-closed set $\K^n$. It is defined over $F$, and its $\delta$-coordinate ring over $F$  is $F\{y_1,\dots,y_n\}$.
\end{example}
\begin{definition}
Let $W \subset \K^n$ be a $\delta$-closed set defined over $F$. 
Let  $\I(W) = \mathfrak{p}_1\cap\ldots\cap \mathfrak{p}_q$ be  a minimal $\delta$-prime decomposition of\ \ $\I(W)$, that is, the $\mathfrak{p}_i \subset F\{y_1,\dots,y_n\}$ are prime $\delta$-ideals containing $ \I(W)$ and minimal with this property. This decomposition is unique  up to permutation (see \cite[Section~VII.29]{Kapldiffalg}).  The irreducible Kolchin-closed sets 
$W_i=\bold{V}(\mathfrak{p}_i)$ are defined over $F$ and  called the {\it irreducible components} of $W$. We  have $W = W_1\cup\ldots\cup W_q$. 
\end{definition}

\begin{definition}
Let $W_1 \subset \K^{n_1}$ and $W_2 \subset \K^{n_2}$ be two Kolchin-closed sets defined over $F$. 
A $\delta$-polynomial map (morphism) defined over $F$ is a map  $$\varphi : W_1\to W_2,\quad a \mapsto \left(f_1(a),\dots,f_{n_2}(a)\right),\ \  a \in W_1\,,$$ where $f_i \in F\{y_1,\dots,y_{n_1}\}$ for all $i=1,\dots,n_2$.  

If $W_1 \subset W_2$, the inclusion map of $W_1$ in $W_2$ is a $\delta$-polynomial map. In this case, we say that $W_1$ is 
a $\delta$-closed subset of $W_2$.
\end{definition}

\begin{example}
Let $\GL_n \subset \K^n$ be the group of $n \times n$ invertible matrices   with entries in $\K$. One can see
$\GL_n$ as a Kolchin-closed subset of $\K^{n^2} \times \K$ defined over $F$, defined by the equation $\det(X)y-1$ in $F\big\{\K^{n^2} \times \K\big\}=F\{X,y\}$, where $X$ is an $n \times n$-matrix of $\delta$-indeterminates over $F$ and $y$ a $\delta$-indeterminate over $F$. One can thus identify the $\delta$-coordinate ring of $\GL_n$ over $F$ with  $F\{Y,1/\det(Y)\}$, where $Y=(y_{i,j})_{1 \leq i,j \leq n} $ is
a matrix of $\delta$-indeterminates over $F$. We also denote  the special linear group that consists of the matrices  of determinant $1$ by $\SL_n \subset \GL_n$.

Similarly, if $V$ is a finite-dimensional $F$-vector space,  $\GL(V)$ is defined as the group of invertible   $\K$-linear maps of $V\otimes_F \K$. To simplify the terminology, we will also treat 
$\GL(V)$ as Kolchin-closed sets tacitly assuming that some basis of $V$ over $F$ is fixed. 
\end{example}

\begin{remark}
If  $K$ is a field, we denote  the group of invertible matrices with coefficients in $K$ by  $\GL_n(K)$. 
\end{remark}

\begin{definition}[{\cite[Chapter~II, Section~1, page~905]{cassdiffgr}\label{def:LDAG}}] A linear differential algebraic group $G \subset \K^{n^2}$ defined  over $F$
is a subgroup of   $\GL_n$ that is a Kolchin-closed set defined over $F$. If $G \subset H \subset \GL_n$ are Kolchin-closed subgroups of 
$\GL_n$, we say that $G$ is a $\delta$-closed subgroup, or $\delta$-subgroup of $H$.
\end{definition}

\begin{proposition}\label{propo:defzariksidenserelationalggroupanddiffalggroup}
Let $G \subset \GL_n$ be a linear algebraic group defined over $F$. We have:
\begin{enumerate}
\item $G$ is a linear differential algebraic group. 
\item Let $H \subset G$ be a $\delta$-subgroup of $G$ defined over $F$, and the Zariski closure $\overline{H} \subset G$ be the  closure of $H$
with respect to the Zariski topology. In this case, $\overline{H}$ is a linear algebraic group defined over $F$, whose polynomial defining ideal over $F$ is 
 $$\I(H) \cap F[Y] \subset \I(H) \subset  F\{Y\}\,,$$ 
 where  $Y=(y_{i,j})_{1 \leq i,j \leq n} $ is
a matrix of $\delta$-indeterminates over $F$.
\end{enumerate}
\end{proposition}

\begin{definition}
Let $G$ be a linear differential algebraic group  defined over $F$. The irreducible component of $G$ containing the identity element $e$  is called the {\it identity component} of $G$ and denoted by $G^\circ$. The linear differential algebraic group $G^\circ$ is a $\delta$-subgroup of $G$ defined over $F$. The linear differential algebraic group
$G$ is said to be {\it connected} if $G = G^\circ$, which is equivalent to $G$ being an irreducible Kolchin-closed set \cite[page~906]{cassdiffgr}.
\end{definition}

\begin{definition}[{\cite{CassidyRep},\cite[Definition~6]{Tandifgrov}}]\label{defn:diffrep} Let $G$ be a linear differential algebraic group defined over $F$ and let $V$ be  a
finite-dimensional vector space over $F$. A  $\delta$-polynomial
group homomorphism  $\rho : G \to \GL(V)$ defined over $F$ is called a
{\it representation} of $G$ over $F$. We shall also say that $V$ is a \emph{$G$-module} over $F$. By a faithful (respectively, simple,  semisimple) $G$-module, we mean
a faithful (respectively, irreducible, completely reducible) representation $\rho:G \rightarrow \GL(V)$.
\end{definition}

The image of a $\delta$-polynomial group homomorphism $\varrho : G\to H$ is Kolchin closed \cite[Proposition~7]{cassdiffgr}. Moreover, if $\ker(\varrho)=\{e\}$, then $\rho$ is an isomorphism of linear differential algebraic groups between $G$ and $\rho(G)$ \cite[Proposition~8]{cassdiffgr}.

\begin{definition}[{\cite[Theorem~2]{Cassunip}}] A linear differential algebraic group $G$ is {\it unipotent} if one of 
the following equivalent conditions holds:
\begin{enumerate}
\item $G$ is conjugate to a differential algebraic subgroup  of the group  of unipotent upper triangular matrices;
\item $G$ contains no elements of finite order $>1\,$;
\item $G$ has a descending normal sequence  of differential algebraic subgroups 
$$G=G_0 \supset G_1 \supset \ldots \supset G_N =\{e\}$$
with $G_i/G_{i+1}$ isomorphic to a differential algebraic subgroup of the additive group $\bold{G}_a$.
\end{enumerate}
\end{definition}

One can  show that 
a linear differential algebraic group   $G$ defined over $F$ admits a  largest normal unipotent differential algebraic subgroup defined over $F$ \cite[Theorem~3.10]{diffreductive}.

\begin{definition} Let $G$ be a linear differential algebraic group defined over $F$. The  largest normal unipotent differential algebraic subgroup of $G$ defined over $F$ is  called the {\it unipotent radical} of $G$ 
and denoted by $R_u(G)$. The unipotent radical of a linear algebraic group $H$ is also denoted by $R_u(H)$.
\end{definition}

Note that,  for a linear differential algebraic group $G$,  we always have
$$\overline{R_u(G)} \subset R_u(\overline{G})$$ and this inclusion can be strict \cite[Example~3.17]{diffreductive}.

\subsubsection{Almost direct products and reductive linear differential algebraic group}
We  recall what  reductive linear differential algebraic groups are and how they decompose into almost direct products of tori and quasi-simple subgroups. 

\begin{definition}A linear differential algebraic group $G$ is said to be {\em simple} if $\{e\}$ and $G$ are the only normal differential algebraic subgroups of $G$.
\end{definition}

\begin{definition}
A \emph{quasi-simple} linear (differential) algebraic group is a finite central extension of a simple non-commutative linear (differential) algebraic group.
 \end{definition}
 
\begin{definition}[{\cite[Definition~3.12]{diffreductive}}]
A linear differential algebraic group $G$ defined over $F$ is said to be {\it reductive} if  $R_u(G) = \{e\}$.
\end{definition}

By definition, the following holds for linear differential algebraic groups: $$\text{simple}\implies\text{quasi-simple}\implies\text{reductive}.$$

\begin{example}$\SL_2$ is quasi-simple but not simple, while $\PSL_2$ is simple.
\end{example}

\begin{proposition}[{\cite[Remark~2.9]{MinOvSing}}]
Let $G \subset \GL_n$ be a linear differential algebraic group defined over $F$. If $\overline{G} \subset \GL_n$
is a reductive linear algebraic group, then $G$ is a reductive linear differential algebraic group.
\end{proposition}

\begin{proposition}\label{prop:cr}
Let $G\subset\Gl(V)$ be a linear differential algebraic group. The following statements are equivalent:
\begin{enumerate}
\item the $G$-module $V$ is semisimple;
\item $V$ is semisimple as a $\overline{G}$-module, where $\overline{G}\subset\Gl(V)$ stands for the Zariski closure;
\item $\overline{G}$ is reductive;
\item $V$ is semisimple as a $\overline{G}^\circ$-module;
\item $V$ is semisimple as a $G^\circ$-module.
\end{enumerate}
\end{proposition}
\begin{proof}
For every subspace $U\subset V$,
the set $N$ of elements $g\in\Gl(V)$ preserving $U$ is an algebraic subgroup of $\Gl(V)$. Therefore, $U$ is $G$-invariant if and only if it is $\overline{G}$-invariant: $$G\subset N\Leftrightarrow \overline{G}\subset N.$$ This implies (1)$\Leftrightarrow$(2).
The equivalences (2)$\Leftrightarrow$(3)$\Leftrightarrow$(4) are well-known (see, for example,~\cite[Chapter~2]{SpringerInv}).
Since the Kolchin topology contains the Zariski topology of $\Gl(V)$, $\overline{G^\circ}$ is Zariski irreducible, hence, equals $\overline{G}^\circ$. Applying (1)$\Leftrightarrow$(2) to the case of a connected $G$, we obtain (4)$\Leftrightarrow$(5).
\qed\end{proof}

\begin{definition}
Let $G$ be a group and $G_1,\dots,G_n$ some subgroups of $G$. We say that 
$G$ is the almost direct product of $G_1,\dots,G_n$ if 
\begin{enumerate}
\item the commutator subgroups $[G_i,G_j]=\{e\}$ for all $i \neq j\,$;
\item the morphism $$\psi :G_1\times \hdots \times G_n \rightarrow G,\quad (g_1,\dots,g_n) \mapsto g_1\cdot\hdots \cdot g_n$$
is an isogeny, that is, a surjective map with a finite kernel.  
\end{enumerate}
\end{definition}

We summarize some  results on the decomposition of
reductive, algebraic and differential algebraic, groups  in the theorem below. We refer to Definition~\ref{def:Kc}  
for the 
notation $G(C)$ with $G$ a linear (differential) algebraic group defined over $C$.

\begin{theorem}\label{thm:decompalmostdirectprodreductive}
Let $G \subset \GL_n$ be a linear differential algebraic group defined over $F$. Assume that  
 $\overline{G} \subset \GL_n $ is  a connected reductive algebraic group.
Then
\begin{enumerate}
\item\label{part1} $\overline{G}$ is an almost direct product of a torus $H_0$ and non-commutative 
normal quasi-simple linear algebraic groups $H_1,\dots,H_s\,$ defined over $\mathbb{Q}$;
\item\label{part2} $G$ is an almost direct product of a Zariski dense $\delta$-closed subgroup $G_0$ of $H_0$ and some $\delta$-closed subgroups $G_i$ 
of $H_i$ for $i=1,\dots,s\,$;
 \item\label{part3}  moreover , either $G_i=H_i$ or $G_i$ is conjugate by a matrix of $H_i$ to $H_i(C)\,$;
\end{enumerate}
The $H_i$'s are called the quasi-simple components of $\overline G\,$; the $G_i$'s are called the $\delta$-quasi-simple components 
of $G$. 
\end{theorem} 
\begin{proof}
Part~\eqref{part1} can be found in \cite[Theorem~27.5, page~167]{Humphline}.
Parts~\eqref{part2} and~\eqref{part3} are contained in \cite[proof of Lemma 4.5]{diffreductive} and \cite[Theorems~15 and~18]{Cassimpl}.
\qed\end{proof}
\begin{remark}
As noticed in \cite[Section~5.3.1]{MinOvSing}, the decomposition of $\overline{G}$ as above can be made effective.
\end{remark}

\begin{proposition}\label{prop:tensorprod}
If $\nu: G_1\times G_2\to G$ is a surjective homomorphism of linear differential algebraic groups and $V$ is a simple $G$-module, then $V$, viewed as a $G_1\times G_2$-module via $\nu$, is isomorphic to $V_1\otimes V_2$, where each $V_i$ is a simple $G_i$-module.
\end{proposition}
\begin{proof}
Since $\nu$ is surjective, $V$ is simple as a $G_1\times G_2$-module.
Let $V_1$ be a simple (non-zero) $G_1$-submodule of $V$ and $U\subset V$ the sum of all $G_1$-submodules isomorphic to $V_1$. Since all elements of $G_2$ send $V_1$ to an isomorphic submodule, we obtain that $U$ is $G_1\times G_2$-invariant. Since $V$ is $G_1\times G_2$-simple, $U=V$. We choose a direct sum decomposition $$V =\bigoplus_{j \in J} U_j,\quad U_j \cong V_1\ \ \text{for all}\ \ j \in J,$$ and, for each $j\in J$, a non-zero $u_j \in U_j$, and let $V_2 = \Span_{j\in J}\{u_j\}\subset V$.  We see that, as $G_1$-modules, $V\cong V_1\otimes V_2$, where $G_1$ acts trivially on $V_2$.

By \cite[Exercise 11.30]{Vinberg}, every endomorphism of $V_1\otimes V_2$ commuting with the action of $G_1$ has the form $\id_{V_1}\otimes A$, where $A$ is an endomorphism of $V_2$. This means that $V_2$ has a structure of a $G_2$-module such that the $G_1$-module isomorphism $V\cong V_1\otimes V_2$  extends to a $G_1\times G_2$-module isomorphism. Since $V$ is $G_1\times G_2$-simple, $V_2$ is $G_2$-simple. It remains to note that the representation $G_i\to\Gl(V_i)$, $i=1,2$, is differential since it is isomorphic to a subrepresentation of the representation $G_i\to \GL(V)$. 
\qed\end{proof}

\begin{definition}
A connected linear differential algebraic group $T$ is called a \emph{$\delta$-torus} if there is an isomorphism $\alpha$ of $T$ onto a Zariski dense $\delta$-subgroup $T'\subset\left(\k^\times\right)^n$, $n \geq 0$.
\end{definition}  
Let $T'_C=\left(C^\times\right)^n$. By~\cite[Proposition~31]{cassdiffgr}, $T'_C\subset T'$. Let $T_C=\alpha^{-1}(T'_C)$. The $\delta$-subgroup $T_C$ does not depend on the choice of $\alpha$: since any differential homomorphism $\left(C^\times\right)^n\to\left(\k^\times\right)^m$ is monomial in each of the $m$ components, its image is contained in $\left(C^\times\right)^m$.

\begin{corollary}\label{cor:diag}
Let $G\subset\Gl(V)$ be a connected linear differential algebraic group. If the $G$-module $V$ is simple and non-constant, then there exists a $\delta$-torus $T\subset G$ such that $V$ is semisimple and non-constant as a $T$-module.
\end{corollary}
\begin{proof}
Since $V$ is simple, $G$ is reductive by Proposition~\ref{prop:cr}.
By Theorem~\ref{thm:decompalmostdirectprodreductive}, $G$ decomposes  as an almost direct product of a $\delta$-torus $G_0$ and $\delta$-quasi-simple components $G_i$, $1\leq i\leq s$. By Proposition~\ref{prop:tensorprod}, $V$ is a tensor product of simple $G_i$-modules $W_i$. By \cite[Theorem~3.3]{diffreductive}, representations of $G_i$ on $W_i$ are polynomial, that is, extend to algebraic representations $\rho_i:\overline{G_i}\to\Gl(W_i)$. 

Since $V$ is non-constant, there is an $i$, $0\leq i\leq s$, such that $W_i$ is non-constant. If $i>0$, then $G_i=\overline{G_i}$. Indeed, otherwise $G_i\simeq H(C)$, where $H=\overline{G_i}$ is a quasi-simple algebraic group defined over $C$ (see Theorem~\ref{thm:decompalmostdirectprodreductive}). Since all algebraic representations of $H$ are defined over $\mathbb{Q}$ (see, for example, \cite[Section~5]{BorelLNM}), $\rho_i(G_i)$ is conjugate to constants, which contradicts the assumption on $W_i$. Thus, $G_i=\overline{G_i}$, and we can take $T$ to be a maximal torus of $G_i$ (see \cite[Sections~21.3-21.4]{Humphline}). If $i=0$, let $T=G_0$.
\qed\end{proof}

\subsubsection{Conjugation to constants}
Conjugation to constants will play an essential role in our arguments. We  recall what it means. As before, $\K$ is a differentially closed field containing $F$ and $C$ is the field  of $\delta$-constants of $\K$.
\begin{definition}\label{defn:connjtoconstants}
Let $G\subset\GL_n$ be a linear algebraic group over $F$. We say that $G$ is conjugate to constants if there exists $h \in\GL_n$ such that $hGh^{-1} \subset \GL_n(C)$. Similarly, we say that a representation $\rho: G \rightarrow \Gl_n$  is conjugate to constants if  $\rho(G)$ is conjugate to constants in $\Gl_n$. 
\end{definition}

\begin{proposition}\label{prop:almostdirectconjconstant}
Let $\rho:G\subset \Gl(W) \rightarrow \Gl(V)$ be  a representation   of  a linear differential algebraic group $G$
 such that $\overline{G} \subset \Gl(W)$ is a connected reductive linear algebraic group. Assume that $\rho$
 is defined over the field $C$. 
With notation of Theorem  \ref{thm:decompalmostdirectprodreductive}, assume that
$Z$ acts by constant weights on $V$ and that, for all $i=1,\dots,s$, either $H_i \ne G_i$ or  $\rho|_{H_i}$ is the identity. 
Then there exists $g \in \overline{G}$ such that 
$$\rho\big(gGg^{-1}\big) \subset \Gl(V)(C).$$
\end{proposition}
\begin{proof}
Let $S =\{i \:|\: H_i=G_i \}$.  By assumption, $\rho(H_i)=\{1\}$ for all $i \in S$.
By Theorem \ref{thm:decompalmostdirectprodreductive}, for all $i \notin S$, there exists $g_i \in G_i$ such that
$g_i H_ig_i^{-1} \subset G_i(C)$. Set $$g= \prod_{i\in S} g_i \in G.$$ Let $h \in G$. Since 
$G$ is the almost direct product of $Z$ and of its  $\delta$-quasi-simple components, there exist 
$z \in Z$ and, for $i \in \{1,\ldots,s\}$, an element $h_i \in H_i$ such that $h=z h_1\cdot \ldots\cdot h_s$. Now,
$$
\rho\big(g hg ^{-1}\big) = \rho(z) \prod_{i\notin S} \rho\big(g_i h_i g_i^{-1}\big).
$$
Since $\rho$ is defined over the constants  and $g_ih_ig_i^{-1} \in G_i(C)$ for all $i \notin S$, we find that $$\rho\big(g_ih_ig_i^{-1}\big) \subset \Gl(V)(C).$$ Since $\rho(z)$ is also constant, the same holds for $\rho\big(gh g^{-1}\big)$. \qed\end{proof}

\subsection{Parameterized differential modules}
In this section, we recall the basic definitions of differential modules and prolongation functors for differential modules with parameters. We then continue with the notion of complete integrability of differential modules and its relation to conjugation to constants of parameterized differential Galois groups. We also show a new result, Proposition~\ref{prop:connectedcompconjconst}, which relates the conjugation to constants of a linear differential algebraic group and of its identity component. 
\subsubsection{Differential modules and prolongations}

Let $K$ be  a $\Delta=\{\p,\delta\}$-field. We denote by $\k$ (respectively, $C$) the field of $\p$ (respectively, $\D$)-constants of $K$.
We assume for simplicity that $(\k,\delta)$ is {\bf differentially closed} (this assumption was relaxed in \cite{GGO,wibdesc,LSN}).
Therefore, unless  explicitly mentioned, any Kolchin-closed set considered in the rest of the paper is a subset of some $\k^n$.
\begin{definition}
  A $\p$-module
$ \cM $ over $K$ is a left $K[\p]$-module that is a finite-dimensional vector space over $K$.
\end{definition}

Let $\cM$ be a $\p$-module over $K$
and let  $\{e_1,\dots,e_n\}$ be a $K$-basis of $\cM$. Let $A=(a_{i,j}) \in K^{n \times n}$ be the matrix defined
by 
\begin{equation}\label{eq:dm}
\p (e_i)= - \sum_{j=1}^n a_{j,i} e_j,\quad i=1,\dots,n.
\end{equation} Then, for any element $m=\sum_{i=1}^n y_i e_i$, where 
$Y=(y_1,
 \ldots ,
 y_n)^T \in K^n$, we have $$\p(m) = \sum_{i=1}^n \p(y_i)e_i - \sum_{i=1}^n \left( \sum_{j=1}^n a_{i,j} y_j\right)e_i.$$ 
 Thus, the equation $\p(m)=0$ translates into the  linear differential system $\p(Y)=AY$. 
 \begin{definition}\label{defn:systmodule}
Let $\cM$  be a $\p$-module over $K$
and  $\{e_1,\dots,e_n\}$ be a $K$-basis of $\cM$.
 We say that the linear differential
 system $\p(Y)= AY$, as above,  is associated  to the  $\p$-module $\cM$ (via the choice of a $K$-basis).  
 Conversely, to a given  
 linear differential system $\p(Y)=AY$, $A=(a_{i,j}) \in K^{n\times n}$, one associates a $\p$-module
 $\cM$ over $K$, namely $\cM=K^n$ with the standard basis $(e_1,\dots,e_n)$ and action of $\p$ given by~\eqref{eq:dm}.
  \end{definition}

Another choice of a $K$-basis $X = BY$, where $B \in \GL_n(K)$,
 leads to the differential system $$\p(X) = (B^{-1}AB-B^{-1}\p(B))X.$$
\begin{definition}
We say that a linear differential system $\p(X)=\tilde{A} X$, with $\tilde{A} \in K^{n \times n}$, is $K$-equivalent (or gauge equivalent over $K$) to a linear differential system $\p(X)=AX$, with  $A \in K^{n \times n}$, if there exists $B \in \GL_n(K)$ such that $$\tilde{A}=B^{-1}AB-B^{-1}\p(B).$$
  \end{definition}

 One has the following correspondence between linear differential systems and linear differential equations.  
 For $L=\p^n +a_{n-1}\p^{n-1}+\ldots +a_0 \in K[\p]$, one can consider the  companion matrix
 $$
 A_L=\begin{pmatrix}
 0&1& \hdots & 0 \\
 0&\ddots&\ddots& \vdots \\
 \vdots &\ddots &0 & 1 \\
 -a_0& -a_1& \dots &-a_{n-1} 
 \end{pmatrix}.
 $$
 The differential system $\p Y=A_LY$ induces a $\p$-module structure
 on $K^n$, which we denote by $\mathcal L$. Conversely, the Cyclic vector lemma \cite[Proposition~2.9]{vdPutSingerDifferential} states
 that any $\p$-module is isomorphic to a $\p$-module $\mathcal L$, of the above form, provided $\k\subsetneq K$.

  \begin{definition} A
 morphism  of $\p$-modules over $K$ is  a  homomorphism of $K[\p]$-modules. 
 \end{definition}

One can  consider  the category  $\Diff_K$ of $\p$-modules over $K$:

 \begin{definition}
We can define the following constructions in $\Diff_K$:
\begin{enumerate}
\item The direct sum of two $\p$-modules, $\cM_1$ and $\cM_2$, is $\cM_1 \oplus \cM_2$ together with the action of
$\p$ defined by $$\p(m_1\oplus m_2)= \p(m_1)\oplus \p( m_2).$$ 
\item The tensor product of two $\p$-modules, $\cM_1$ and $\cM_2$, is $\cM_1 \otimes_K \cM_2$ together with the action of
$\p$ defined by $$\p(m_1 \otimes m_2)= \p(m_1)\otimes m_2 +m_1 \otimes \p( m_2).$$ 
\item
The unit object $\bold{1}$ for the tensor product  is
the field $K$ together with  the left $K[\p]$-module structure given by $$(a_0+ a_1 \p +\dots +a_n\p^n)(f)= a_0f+\dots+ a_n \p^n(f)$$ for $f,a_0,\dots,a_n \in K$. 
\item The internal Hom of  two $\p$-modules $\cM_1, \cM_2$ exists in $\Diff_K$ and 
is denoted by $\underline{\Hom}(\cM_1,\cM_2)$. It consists of the $K$-vector space $\Hom_K(\cM_1,\cM_2)$
of  $K$-linear maps from $\cM_1$ to $\cM_2$ together with the action of $\p$ given by the formula 
$$\p u(m_1)= \p(u(m_1))-u(\p m_1).$$ The dual $\cM^ *$ of a $\p$-module $\cM$ is the $\p$-module $\underline{\Hom}(\cM,\bold{1})$.
\item An endofunctor  $D : \Diff_K\to \Diff_K$, called the prolongation functor, is defined as follows: 
 if $\cM$ is an object of $\Diff_K$ corresponding to the
linear differential system $\p (Y)=A Y$, then $D(\cM)$ corresponds to the linear differential system $$\p(Z)=\begin{pmatrix} A & \delta(A) \\ 0& A \end{pmatrix} Z.$$
\end{enumerate}
\end{definition}

The construction of the prolongation functor  reflects
the following idea. If $U$ is a fundamental solution matrix of $\p(Y)=AY$
in some $\Delta$-field extension  $F$ of $K$, that is, $\p(U)=AU$ and $U \in \GL_n(F)$, then
$$
 \p(\delta U) =\delta(\p U)=\delta(A) U +A \delta (U).
$$
Then, $\begin{pmatrix} U & \delta(U) \\ 0& U \end{pmatrix}$ is a fundamental solution matrix of $\p(Z)=\begin{pmatrix} A & \delta(A) \\ 0& A \end{pmatrix} Z$.
 Endowed with all these constructions, it follows from \cite[Corollary~3]{OvchTannakian} that the category $\Diff_K$ is  a {\em $\delta$-tensor category} (in the sense of \cite[Definition~3]{OvchTannakian} and \cite[Definition~4.2.1]{MosheTAMS}).

In this paper, we will not  consider the whole category $\Diff_K$ but the $\de$-tensor subcategory generated by a 
$\p$-module. More precisely, we have the following definition.  

\begin{definition}
 Let $\cM$ be an object of $\Diff_K$. We denote   by $\{\cM\}^{\otimes, \delta}$ 
 the smallest full subcategory of $\Diff_K$ that 
contains $\cM$ and is closed under all operations of linear algebra (direct sums, tensor products, duals, and subquotients)
and under $D$. The category  $\{\cM\}^{\otimes, \delta}$ is  a $\delta$-tensor category over~$\k$. We also  denote   by $\{\cM\}^{\otimes}$ the full tensor 
subcategory of $\Diff_K$ generated by $\cM$.   Then, $\{\cM\}^{\otimes}$ is a tensor category over~$\k$.
\end{definition}

Similarly, the category $\Vect_\k$ of finite-dimensional 
$\k$-vector spaces is a $\delta$-tensor category.  The prolongation functor on $\Vect_\k$ is defined as follows: for  a $\k$-vector space
$V$, the $\k$-vector space $D(V)$  equals  $\k[\delta]_{\leq 1}\otimes _\k V$, where $\k[\delta]_{\leq 1}$ is considered as the right $\k$-module
of $\delta$-operators up to order $1$ and $V$ is viewed as a left $\k$-module.

\begin{definition}\label{defn:fiberfunctor}
Let  $\cM$ be an  object of $\Diff_K$.  A $\delta$-fiber functor $\o :\{\cM\}^{\otimes, \delta} \rightarrow \Vect_\k$ is
an exact, faithful, $\k$-linear, tensor compatible functor  together with a natural isomorphism  between $D_{\Vect_\k} \circ \omega$ and 
$\omega\circ D_{ \{\cM\}^{\otimes, \delta}}$ \cite[Definition~4.2.7]{MosheTAMS}, where the subscripts emphasize the category on which  we perform the prolongation. The pair $\big(\{\cM\}^{\otimes, \delta}, \omega\big)$ is called a $\delta$-Tannakian category.
\end{definition}

\begin{theorem}[{\cite[Corollaries~4.29 and~6.2]{GGO}}]
Let $\cM$ be an object of $\Diff_K$.
Since $\k$ is $\delta$-closed,  the category $\{\cM\}^{\otimes, \delta}$ admits a $\delta$-fiber functor and 
 any two 
 $\delta$-fiber functors are naturally isomorphic.
\end{theorem}

\begin{definition}\label{defn:difftannakautogroup}
Let $\cM$ be an object of $\Diff_\k$ and $\o :\{\cM\}^{\otimes, \delta} \rightarrow \Vect_\k$ be a $\delta$-fiber functor. 
The group $\Gal^\delta(\cM)$ of $\delta$-tensor isomorphisms  of $\omega$ is defined as follows. It consists
of  the elements $ g\in\GL(\o(\cM))$ that stabilize $\omega(\mathcal{V})$ for  every $\p$-module  $\mathcal{V}$ obtained from $\cM$ by applying the linear constructions (subquotient, direct sum, tensor product, and dual), and the 
prolongation functor. The action of $g$ on $\omega(\mathcal V)$ is obtained by applying the same constructions to $g$. We call $\Gal^\delta(\cM)$ the parameterized differential Galois group of $(\cM,\omega)$, or of $\cM$ when there is no confusion.\end{definition}

\begin{theorem}[{\cite[Theorem~2]{OvchTannakian}}]\label{thm:difftanequ}
Let $\cM$ be an object of $\Diff_K$ and $\o :\{\cM\}^{\otimes, \delta} \rightarrow \Vect_\k$ be a $\delta$-fiber functor. The group  $\Gal^\delta(\cM) \subset \GL(\o(\cM))$ is a linear differential algebraic group defined over $\k$,  and $\o$ induces an equivalence of  categories between $\{\cM\}^{\otimes, \delta}$
and the category of finite-dimensional  representations of $\Gal^\delta(\cM)$.
\end{theorem}
\begin{definition}\label{defn:trivial objects}
We say that a $\p$-module $\cM$ over $K$ is \textit{trivial} if it is either $(0)$ or isomorphic as $\p$-module over $K$ to  $\bold{1}^n$ for 
some positive integer $n$. For $G$ a linear differential algebraic group over $\k$, we say that 
a $G$-module $V$ is \textit{trivial} if   $G$ acts identically  on $V$. 
\end{definition}
\begin{remark}
For $\cM$ an object of $\Diff_K$ and $\o :\{\cM\}^{\otimes, \delta} \rightarrow \Vect_\k$  a $\delta$-fiber functor, the following holds:
a $\p$-module $\cN $ in $\{\cM\}^{\otimes, \delta}$ is trivial if and only if $\o(\cN)$ is a trivial $\Gal^\delta(\cM)$-module.
\end{remark}

\begin{remark}
The parameterized differential Galois group depends a priori on the choice of a $\delta$-fiber functor $\o$.  However, since  two  $\delta$-fiber functors for $\{\cM\}^{\otimes,\delta}$ are naturally isomorphic, we find  that the parameterized differential Galois groups that these functors define are isomorphic as linear differential algebraic groups over $\k$. Thus, if it is not necessary, we will speak of the parameterized differential Galois group of $\cM$ without mentioning the 
$\delta$-fiber functor.
\end{remark}

Forgetting the action of $\delta$,  one can similarly define the group $\Gal(\cM)$ of tensor isomorphisms of $\omega : \{\cM\}^{\otimes} \rightarrow \Vect_\k$. By \cite{Deligne:categoriestannakien}, the group $\Gal(\cM) \subset \GL(\o(\cM))$ is a linear algebraic group defined over $\k$, and $\omega$ induces an equivalence of categories between $\{\cM\}^{\otimes}$ and the 
category of  $\k$-finite-dimensional representations of $\Gal(\cM)$. We call  $\Gal(\cM)$ the {\em differential Galois group} of 
$\cM$ over~$K$.

\begin{proposition}[{\cite[Proposition~6.21]{HardouinSinger}}]\label{prop:zariskidensegalois groups}
If $\cM$ is an object of $\Diff_K$ and $\o :\{\cM\}^{\otimes, \delta} \rightarrow \Vect_\k$ is a $\delta$-fiber functor, then $\Gal^\delta(\cM)$ is a Zariski dense subgroup of $\Gal(\cM)$ (see Proposition \ref{propo:defzariksidenserelationalggroupanddiffalggroup}).
\end{proposition}

\begin{definition} A {\em parameterized Picard--Vessiot extension}, or {\em PPV extension} for short,  of $K$ for a $\p$-module $\mathcal M$ over $K$
is a $\Delta$-field extension $K_\cM$ that is generated over $K$ by the entries of a fundamental solution matrix $U$ of a differential system $\p(X)=AX$
associated to $\cM$   and such that $K_\cM^\partial = K^\partial$. The field $K(U)$ is a {\em Picard--Vessiot extension} ({\em PV extension} for short), that is, a $\p$-field extension of $K$ generated
by the entries of a fundamental solution matrix $U$ of $\p(X)=AX$ such that  $K(U)^\p=K^\p$.
\end{definition}

A parameterized Picard--Vessiot extension associated to a $\p$-module $\cM$ depends a priori on the choice of a $K$-basis of $\cM$, which is equivalent to the choice of a linear differential system associated to $\cM$. However, one can show that gauge equivalent differential systems lead to parameterized Picard--Vessiot extensions that are isomorphic as $K$-$\Delta$-algebras.
In \cite{Deligne:categoriestannakien},  Deligne showed that a fiber functor corresponds to a Picard--Vessiot extension; it is shown in \cite[Theorem~5.5]{GGO}  that the notions of $\delta$-fiber functor and parameterized Picard--Vessiot extension are  equivalent.

\begin{definition}
Let $\cM$ be a $\p$-module over $K$. Let $\p(X)=AX$ be a differential system associated to $\cM$ over $K$ with $A \in K^{n \times n}$ and let $K_\cM$ be a PPV extension for  $\p(X)=AX$ over $K$.
The {\em parameterized Picard--Vessiot group}, or {\em PPV-group} for short is denoted by $\Galdelta(K_\cM/K)$ and  is the set of $\Delta$-automorphisms of $K_\cM$ over $K$,
whereas the {\em Picard--Vessiot group} (usually called the differential Galois group in the literature) of $K_\cM$ over $K$, by definition, is  the set of $\partial$-automorphisms of a Picard--Vessiot extension $K(U)$
of $K$ in $K_\cM$, where $U \in \GL_n(K_\cM)$ is a fundamental solution matrix of $\partial(X)=AX$. This group is denoted by $\Gal(K_\cM/K)$.
\end{definition}

\begin{remark}
Let $U \in \GL_n(K_\cM)$ be a fundamental solution matrix of $\p(X)=AX$. For any $\tau \in \Galdelta(K_\cM/K)$, there exists 
$[\tau]_U \in \GL_n(\k)$ such that $\tau(U)=U[\tau]_U$. The map $$
 \Galdelta(K_\cM/K) \rightarrow \GL_n,\quad \tau \mapsto [\tau]_U
$$
is an embedding and identifies $\Galdelta(K_\cM/K)$ with a $\delta$-closed subgroup of $\GL_n$. One can  show that another choice
of fundamental solution matrix as well as another choice of gauge equivalent linear differential system yield a conjugate subgroup in $\GL_n$.
Similarly, one can represent $\Gal(K_\cM/K)$ as a linear algebraic subgroup of $\GL_n$. With  these representations of the Picard--Vessiot groups, one can show that Picard--Vessiot groups and differential Galois groups  are isomorphic in the parameterized and non-parameterized cases.
\end{remark}
In the PPV theory,  a  Galois correspondence holds between
differential algebraic subgroups of the PPV-group and $\Delta$-sub-field extensions of $K_\cM$ (see \cite[Theorem~6.20]{HardouinSinger} for more details). Moreover, the $\delta$-dimension of $\Galdelta(\cM)$ coincides with the $\delta$-transcendence degree of $K_\cM$ over $K$ (see  \cite[page~374 and Proposition~6.26]{HardouinSinger} for the definition of the $\delta$-dimension and $\delta$-transcendence degree and the proof of their equality). Moreover, the defining equations of the parameterized differential Galois group reflect
the differential algebraic relations among the solutions (see \cite[Proposition~6.24]{HardouinSinger}). Therefore, given a $\p$-module $\cM$ over $K$, 
we find that the defining equations of the parameterized  differential Galois group $\Gal^\delta(\cM)$ over $\k$ determine the differential algebraic relations between the solutions in $K_\cM$ over $K$.

\begin{definition}
A $\p$-module $\cM$ is said to be completely reducible if, for every $\p$-submodule $\cN$ of $\cM$, there exists
a $\p$-submodule $\cN'$ of $\cM$ such that $\cM=\cN \oplus \cN'$. We say that a $\p$-operator is completely reducible if 
the associated $\p$-module is completely reducible.  
\end{definition}
By \cite[Exercise 2.38]{vdPutSingerDifferential},
a $\p$-module is completely reducible if and only if its  differential Galois group is a reductive linear algebraic group.
Moreover, for  a completely reducible $\p$-module $\cM$, any object in $\{\cM\}^\otimes$ is completely
reducible.

\subsubsection{Isomonodromic differential modules}

\begin{definition}[{\cite[Definition 3.8]{cassisinger}}]
Let $A \in K^{n \times n}$. We say that the linear differential system $\partial Y=AY$ is isomonodromic (or completely integrable) over $K$ if there 
exists $B \in K^{n\times n}$ such that $$ \p(B)-\delta(A)=AB -BA.$$ \end{definition}
\begin{remark}
One can show that a linear differential system $\p Y =A Y$ 
is isomonodromic if and only if there exists a $\Delta$-field extension $L$ of $K$  
and $B \in K^{n\times n}$ such that the system 
$$ 
\left\{ 
\begin{aligned}
\p Y & = AY \\
\delta Y & =BY
\end{aligned} \right.
$$
has a fundamental solution matrix with coefficients in $L$.
\end{remark}

We recall a characterization of complete integrability in terms of the PPV theory.

\begin{proposition}[{\cite[Proposition 3.9]{cassisinger}}]\label{prop:compintconjconts}
Let $\cM$ be a $\p$-module over $K$ and $\p(Y)=A Y$, with $A \in K^{n\times n}$, be an associated linear differential system.
The following statements are equivalent:
\begin{itemize}
\item  $\Galdelta(\cM)$ is conjugate to constants in $\GL(\o(\cM))$ 
 (see Definition \ref{defn:connjtoconstants});
\item The linear differential system $\p(Y)=AY$ is isomonodromic over $K$.  
\end{itemize} 
 \end{proposition}

The proof of the following result was provided to the authors by Michael F. Singer and will be used in the proof of Proposition~\ref{prop:connectedcompconjconst}.
\begin{lemma}\label{lem:inverse}
Given a linear differential algebraic group $G\subset \GL_n$ defined over a differentially closed  field $(\k,\dt)$ and any $\Delta = \{\dx,\dt\}$-field $K$ such that $K^\p=\k$,  there exists a $\Delta$-field extension $F$ of $K$ such that $F^\p=\k$ and  $G$ can be realized as a parameterized differential Galois group over $F$ in the given faithful representation of $G \subset \GL_n $.
\end{lemma}
\begin{proof} We first consider the ``generic'' case: 
we construct a $\Delta$-field extension $E$ of $K$ with no new $\dx$-constants such that $\GL_n$ is a parameterized differential Galois group of a $\partial$-module $\cM$ over $E$. Assume we have constructed $E$ and let $E_\cM$ be a PPV extension of $\cM$ over $E$.
 For any differential algebraic subgroup $G$ of $\GL_n$, let $F$ be the fixed field of $G$ in $E_\cM$, i.e., the elements of $E_\cM$ fixed by $G$.  By the PPV correspondence,  $G$ is the parameterized differential Galois group of $E_\cM$ over $F$. Moreover, $$K^\p=\k \subset F^\partial\subset E_\cM^\partial = \k.$$
To construct the fields $E_\cM$ and $E$ for $\GL_n$, we shall follow the construction from \cite[pages~87--89]{Magid}.
Let $\{z_{i,j}\}$ be a set of $n^2$ $\Delta$-differential indeterminates over $K$. Let $E_\cM = K\langle z_{i,j}\rangle_\Delta$ be a $\Delta$-field of differential rational functions in these indeterminates. Note that the $\dt$-constants of $E_\cM$ are $\k$, as in \cite[Lemma~2.14]{Magid}. Let $Z = (z_{i,j}) \in \GL_n(E_\cM)$ and 
$A = (\dx Z)(Z)^{-1}$.
We then have that 
\begin{equation} \label{eqn1} \dx Z  = AZ.
\end{equation}
 Let $E$ be the $\Delta$-field generated over $K$ by the entries of $A$.   Then, $E_\cM$ is a PPV extension of $E$ for equation (\ref{eqn1}). Since $Z$ is a matrix of $\Delta$-differential indeterminates, any assignment $Z \mapsto Zg$ for $g \in \GL_n(K)$ defines a $\Delta$-$K$-automorphism $\phi_g$  of $E_\cM$ over $K$.  If we restrict to those $g \in \GL_n=\GL_n(\k)$, then $\phi_g$ leaves $A$ fixed and so  all elements of $E$ are left fixed.  Therefore, $\GL_n$ is a subgroup of the PPV-group of $E_\cM$ over $E$.  Since  this PPV-group is already a subgroup of $\GL_n$, we must have that the PPV-group  of $E_\cM$ over $E$ is $\GL_n$.
\qed\end{proof}
The proof of the following result  uses PPV theory, which does not appear in the statement. It is, therefore, of interest to find a direct proof of it as well.

\begin{proposition}\label{prop:connectedcompconjconst}
Let $G \subset \Gl(V)$ be a linear differential algebraic group over $\k$ and let 
$G^{\circ}$ be the identity
component of $G$. If $G^\circ$ is conjugate to  constants  in $\Gl(V)$, then the same holds for $G$.
\end{proposition}
\begin{proof}
By Lemma~\ref{lem:inverse}, let $K$ be a $\Delta$-field with $K^\p=\k$ such that $G$ is a parameterized differential Galois group of a $\p$-module $\cM$ over $K$ and the embedding $G\subset\GL(V)$ is the faithful representation $G \to \GL(\omega(\cM))$.
Let $L/K$ be a PPV  extension for $\cM$ over $K$. 
One can identify $G$ with $\Galdelta(L/K)$, the group of automorphisms of $L$ over $K$ commuting with $\delta$ and $\p$. 
 Let $F$ be the subfield of $L$ fixed by $G^\circ$. By the PPV correspondence \cite[Theorem 9.5]{cassisinger},
 the group of automorphisms of $L$ over $F$ commuting with $\{\delta,\p\}$  
 coincides with $G^\circ$
 and the extension  $F/K$ is algebraic since $G/G^\circ$ is finite.
 
  Let $\p(Y)=AY$ be a linear differential system
associated to  $\cM$. The parameterized differential Galois group of $\cM$ over $F$ is $G^\circ$ and thus conjugate to constants by assumption. Proposition \ref{prop:compintconjconts}
 implies that $\p(Y)=AY$ is isomonodromic over $F$, that is, there exists 
 $B\in F^{n \times n }$ such that 
 \begin{equation}\label{eq:compint}
 \p(B)-\delta(A)=AB -BA.
 \end{equation}
Let $K_0$ be the subfield extension of $F$ generated over $K$ by the coefficients of the matrix $B$. 
Without loss of generality, we can assume that $K_0/K$ is a finite Galois extension in the classical sense. 
We denote by $\Gal(K_0/K)$ its differential Galois group and by $r$ its degree.
By  \cite[Exercise 1.24]{vdPutSingerDifferential}, there exist unique derivations, still denoted $\p$ and $\delta$ extending  $\p$ and $\delta$
to $K_0$. Moreover, any element of $\Gal(K_0/K)$ commutes with the action of $\delta$ and $\p$ on $K_0$. If we let 
$$C=\frac{1}{r}\sum_{\tau \in \Gal(K_0/K)} \tau(B),$$ 
then $C$ has coefficients in $K$ and satisfies
\begin{align}\label{eq:compint1}
 &\p(A)-\delta(C)=\p(A) - \frac{1}{r} \left(\sum_{\tau \in \Gal(K_0/K)} \tau( \delta(B))\right)\notag\\
 &= \p(A) - \frac{1}{r}\left(\sum_{\tau \in \Gal(K_0/K)} \tau \left( \p(A)-BA + AB\right)\right)= \p(A)-\p(A) +CA -AC.
 \end{align}
 This shows that $\p(Y)=AY$ is  isomonodromic over $K$. By Proposition \ref{prop:compintconjconts}, we find 
 that $G$ is conjugate to constants in $\GL_n$.
 \qed\end{proof}

\section{Calculating the parameterized differential Galois group of $L_1(L_2(y))=0$}\label{sec:main}
In this section, given two completely reducible $\p$-modules $\cL_1$ and  $\cL_2$,  we study the parameterized differential Galois group of an arbitrary $\p$-module extension $\cU$ of $\cL_1$ by $\cL_2$. In Section~\ref{subsec:compunip}, we describe $\Gal^\delta(\cU)$ as a  semi-direct product of a 
$\delta$-closed subgroup of $\Hom(\o(\cL_1),\o(\cL_2))$ by the parameterized differential Galois group $\Gal^\delta(\cL_1 \oplus \cL_2)$ (see Theorem \ref{thm:structuregaloisgroupsemidirectproduct}). In Section~\ref{subsec:red}, we perform a first reduction that allows us to set $\cL_1$ equal to the trivial $\p$-module $\bold{1}$.

In Theorem \ref{thm:redparamradunip}, we show how one can recover a complete description of the parametrized differential Galois group of $\cU$ from the knowledge of the parametrized differential Galois group of its reduction. In Section~\ref{sec:computation},
we thus focus on the computation of the parameterized differential Galois group of an arbitrary $\p$-module extension $\cU$ of $\bold{1}$ by a completely reducible
$\p$-module $\cL$. 

We then show that one can decompose $\cL$ in a ``constant'' and a ``purely non-constant'' part. This decomposition
yields  a decomposition of $R_u(\Gal^\delta(\cU))$. For $K = \k(x)$, the computation of $\Gal^\delta(\cU)$  for the ``constant part'' can be deduced from  the algorithms contained in \cite{MinOvSingunip}, whereas the computation of the ``purely non-constant'' part results from Section~\ref{sec:purelynonconstantcase} and Theorem \ref{thm:purelynonconstantuniprad}. Finally, in  Section~\ref{sec:mergeconstantnonconstant}, we show, under some assumption on $\cL$, that $R_u(\Gal^\delta(\cU))$ is the product of the ``constant'' and ``purely  non-constant'' parts (see Theorem \ref{thm:decomp}).

Throughout this section, $K$ is a $(\delta,\p)$-field of characteristic zero, whose field of $\p$-constants $\k$ is assumed to be $\delta$-closed. 
We denote also by $C$ the field of $\delta$-constants of $\k$. We fix a $\delta$-fiber functor  $\omega :\Diff_K \to \Vect_\k$  on $\Diff_K$ (see Definition \ref{defn:fiberfunctor}).
Any parameterized differential Galois group in this section shall be computed with respect to $\o$ and is a linear differential algebraic group 
defined over $\k$. Any representation is, unless explicitly mentioned, defined over $\k$.

\subsection{Structure of the parameterized differential Galois group } \label{subsec:compunip}

Let $L_1,L_2 \in K[\p]$ be two completely reducible $\p$-operators, and
 let us denote by $\cL_1$ (respectively, by $\cL_2$) the $\p$-module corresponding
to $L_1(y)=0$ (respectively, $L_2(y)=0$).  The $\p$-module $\cU$ over $K$, corresponding 
to $L_1(L_2(y))=0$, is an extension of  $\cL_1$ by $\cL_2$,
$$ \txymatrix{
0 \ar[r] &\cL_2 \ar[r]^{i} & \mathcal{U} \ar[r]^{p} & \cL_1
\ar[r] & 0},$$
in the category of   $\p$-modules over $K$.

\begin{definition}
For any object $ \cX $ in $\{ \cU\}^{\otimes,\delta}$, we  
 define  $\Stab(\cX)$ (respectively, $\Stabdelta(\cX)$) as  the 
set of (respectively, $\delta$-) tensor automorphisms in $\Gal(\cU)$ (respectively, $\Galdelta(\cU)$) that induce the identity on $\omega( \cX )$. 
\end{definition}

By \cite[II.1.36]{demazuregabriel},   $\Stab(\cX)$ (respectively, $\Stabdelta(\cX)$) is a linear (respectively, differential) algebraic group over $\k$. 
One has also that $\Stabdelta(\cX)$ is Zariski dense in $\Stab(\cX)$. Moreover, we have:
\begin{lemma}\label{lemma:stabnormal}
For any object $ \cX $ in $\{ \cU\}^{\otimes,\delta}$, the group  $\Stabdelta( \cX)$ (respectively, $\Stab(\cX)$) is normal in $\Galdelta(\cU)$ (respectively, $\Gal(\cU)$).
\end{lemma} 
 \begin{proof}
 We prove only the parameterized statement. Let $g \in \Galdelta(\cU)$ and $h \in \Stabdelta(\cX)$. 
 One has to show that $ghg^{-1}$ induces the identity on $\o(\cX)$. It  is sufficient to remark that, by definition, any element of  $\Galdelta(\cU)$ stabilizes $\o(\cX)$.
 \qed\end{proof}

The aim of this section is to  prove the following theorem.
 \begin{theorem}\label{thm:structuregaloisgroupsemidirectproduct} 
If $\cL_1,\cL_2$ are completely reducible $\p$-modules over $K$ and if  $\cU$ is a $\p$-module extension over $K$ of 
 $\cL_1$ by $\cL_2$, then  
 \begin{enumerate}
 \item $\Gal^\delta(\cU)$ is an extension of $\Gal^\delta(\cL_1\oplus \cL_2)$ by a $\delta$-subgroup $W\subset\Hom(\o(\cL_1),\o(\cL_2))$.
 \item $W$ is stable under the action of $\Gal^\delta(\cL_1\oplus \cL_2)$ on $\Hom(\o(\cL_1),\o(\cL_2))$ given by 
 $$g*\phi=g\phi(g^{-1})\quad \text{for any}\ \ (g,\phi)  \in  \Gal^\delta(\cL_1\oplus \cL_2) \times \Hom(\o(\cL_1),\o(\cL_2)).$$ 
 \end{enumerate}
  \end{theorem}

 \begin{remark}
 The parameterized differential Galois group $\Gal^\delta(\cL_1\oplus \cL_2)$ acts on the objects of the $\delta$-tensor category generated by $\o(\cL_1 \oplus \cL_2)$.
 The $\k$-vector space $\Hom(\o(\cL_1),\o(\cL_2))$ belongs to this category, and the action of $\Gal^\delta(\cL_1\oplus \cL_2)$ on $\Hom(\o(\cL_1),\o(\cL_2))$  detailed above is just the description of the  Tannakian representation. 
 \end{remark}
 Before proving this theorem, we need some intermediate lemmas.

\begin{lemma}\label{lemma:stabunip}
The linear differential algebraic group $\Galdelta(\cU)$ is an extension of 
the reductive linear differential algebraic group $\Galdelta(\cL_1 \oplus \cL_2)$
by the  linear differential algebraic group $\Stabdelta( \cL_1 \oplus \cL_2 )$.
\end{lemma}
\begin{proof}
Since  $\{ \cL_1 \oplus \cL_2   \}^{\otimes,\delta}$ is a full $\delta$-tensor subcategory of $\{\cU\}^{\otimes,\delta}$,
the linear differential algebraic group  $\Galdelta( \cL_1 \oplus \cL_2   )$   is a quotient of  $\Galdelta(\cU)$. We denote 
the quotient map by $$\pi : \Galdelta(\cU) \rightarrow \Galdelta(\cL_1 \oplus \cL_2 ).$$ Then  $\ker\pi=\Stabdelta( \cL_1 \oplus \cL_2 )$. Since $\cL_1$ and $\cL_2$ are completely reducible, 
  $ \cL_1 \oplus \cL_2 $ is completely reducible as well. This means that  $\Galdelta(\cL_1 \oplus \cL_2)$
is reductive. Since the latter group is the Zariski closure of $\Galdelta( \cL_1 \oplus \cL_2 )$ in $\Gl(\omega( \cL_1 \oplus \cL_2 ))$,
 \cite[Remark 2.9]{MinOvSing} implies that 
$\Galdelta(\cL_1 \oplus \cL_2)$ is a reductive linear differential algebraic group.
\qed\end{proof}

We will relate $\Stabdelta( \cL_1 \oplus \cL_2 )$ to  
 $R_u(\Galdelta(\cU))$ 
and  describe more precisely the structure of the latter group.
By the exactness of $\omega$,
$\omega(\cU)$ is  an extension of $\omega(\cL_1)$
by $\omega(\cL_2)$ in  the 
category of   representations of $\Galdelta(\cU)$.

\begin{lemma}\label{lemma:unipcocycle}
In the above notation, let 
$s$ be a $\k$-linear section of the exact sequence:
\begin{equation}\label{eqn:exactseq}
\txymatrix{
0 \ar[r] &\omega(\cL_2) \ar[r]^{\omega(i)} &
\omega(\mathcal{U}) \ar[r]^{\omega(p)} & \omega(\cL_1) \ar@{.>}@/^/[l]^{s}
\ar[r] & 0}.
 \end{equation}
We consider the following  map
 $$\zeta_\cU: \Galdelta(\cU) \rightarrow \Hom(\omega(\cL_1),\omega(\cL_2)),\quad g \mapsto \left( x \mapsto  g(s(g^{-1} x))-s(x) \right).$$ 
 Then the restriction of the map
$\zeta_\cU$ to $\Stabdelta(  \cL_1 \oplus \cL_2 )$ is a one-to-one morphism of linear differential algebraic groups. 
Moreover, the linear differential algebraic group $\Stabdelta(  \cL_1 \oplus \cL_2 )$ is abelian and coincides with  $R_u( \Galdelta(\cU))$.
\end{lemma}

\begin{proof}

For all   $g_1, \ g_2 \ \in  \Galdelta(\mathcal{U})$, we have:
\beq \label{eqn:co}
 \zeta_\cU(g_1g_2)(x)=g_1\zeta_ \cU(g_2)(g_1^{-1} x)+\zeta_\cU(g_1)(x).  \eeq
If $g_1,g_2\in \Stabdelta(\cL_1 \oplus\cL_2)$, equation~\eqref{eqn:co} gives
$$\zeta_\cU(g_1g_2)=\zeta_\cU(g_1) +\zeta_\cU
(g_2).$$ This means that $\zeta_\cU$ is a morphism of linear differential algebraic groups from $\Stabdelta(\cL_1 \oplus\cL_2)$ to $\Hom(\omega(\cL_1),\omega(\cL_2))$.

Moreover, let   $\{e_j\}_{j = 1 \dots s}$ (respectively, $\{f_i\}_{i = 1\dots r}$)  be a $\k$-basis of 
$\omega(\cL_2)$ (respectively, $\omega(\cL_1)$). Then  $$\left\{\omega(i)(e_i), s(f_j)\right\}_{ i=1, \dots, s, j=1, \dots r}$$ is a $\k$-basis
of   $\omega(\mathcal{U})$. If  $g \ \in \Stabdelta(\cL_1 \oplus\cL_2 )\cap\ker(\zeta_\cU)$, then 
$g$ induces the identity on $$\left\{\omega(i)(e_i), s(f_j)\right\}_{ i=1, \dots, s, j=1, \dots r}$$ and 
thereby on $\omega(\cU)$.
Therefore, by definition of $ \Galdelta(\cU)$, the element $g$ is the identity element and, therefore, $\ker \left(\zeta_\cU\big|_{ \Stabdelta(\cL_1 \oplus\cL_2)} \right)$ is trivial. 

Since $\Hom(\omega(\cL_1),\omega(\cL_2))$ is abelian, the same holds for  $\Stabdelta(  \cL_1 \oplus \cL_2 )$. Moreover,
 $\Stabdelta(  \cL_1 \oplus \cL_2 )$ is unipotent. Indeed, let  $e$  be the  identity  element in $\Galdelta(\cU)$,  $x\in\o(\cL_1)$, and $g \in  \Stabdelta(  \cL_1 \oplus \cL_2 )$. 
  Since $gs(x)-s(x) \in \o(\cL_2)$, we have $$(g-e)^2(s(x))= (g-e)(gs(x)-s(x))= g(gs(x)-s(x))-(gs(x)-s(x))=0.$$ Reasoning as above, we find 
 that 
$(g-e)^2$ is zero on $\o(\cU)$. By Lemma \ref{lemma:stabnormal}, $\Stabdelta( \cL_1 \oplus \cL_2 )$ is also normal and, hence, must be contained
in $R_u( \Galdelta(\cU))$. By \cite[Theorem 1]{Cassunip}, the image of a unipotent linear differential algebraic group is unipotent. By Lemma \ref{lemma:stabunip},  
$\Stabdelta( \cL_1 \oplus \cL_2 )$ is the kernel of the projection of $ \Galdelta(\cU)$
on the reductive linear differential algebraic group $ \Galdelta( \cL_1 \oplus \cL_2 )$.  It follows that
$R_u( \Galdelta(\cU))$ is contained in   $\Stabdelta(  \cL_1 \oplus \cL_2 )$, which ends the proof.
\qed\end{proof}

\begin{remark}\label{rem:indep}
Since two sections of \eqref{eqn:exactseq} differ by a map from 
$\o(\cL_1)$ to $\o(\cL_2)$, one  sees that, when restricted to $R_u(\Galdelta(\cU))=\Stabdelta(  \cL_1 \oplus \cL_2 )$,
the map $\zeta_\cU$ is independent of the choice of the section.
\end{remark}

By the above lemma,  $R_u( \Galdelta(\cU))$
is an abelian normal subgroup of $ \Galdelta(\cU)$. Since $ \Galdelta(\cL_1 \oplus \cL_2)$
is the quotient of $ \Galdelta(\cU)$ by $R_u( \Galdelta(\cU))$ and $R_u( \Galdelta(\cU))$ is abelian, the linear differential algebraic group
 $ \Galdelta(\cL_1 \oplus \cL_2)$ acts by conjugation on $R_u( \Galdelta(\cU))$.
The lemma below shows that this action is compatible with
the action of $\Galdelta(\cL_1 \oplus \cL_2)$ on $\Hom_\k(\omega(\cL_1),\omega(\cL_2))$.

\begin{lemma}\label{lemma:compact}
For all $g_1  \ \in \Galdelta(\mathcal{U})$, $g_2 \ \in R_u(\Galdelta(\cU))$, and $x \in \omega(\cL_1)$, we have $$\zeta_\cU\big( g_1
g_2 {g_1}^{-1}\big)(x)=g_1\big(\zeta_\cU(g_2)\big(g_1^{-1}x\big)\big)=g_1*\zeta_\cU(g_2)(x),$$
where $*$  denotes the natural action  of $\Gal^\delta(\cL_1\oplus \cL_2)$ on $\Hom(\o(\cL_1),\o(\cL_2))$ via  $$g *\phi=  g \circ  \phi \circ g^{-1}\quad \text{for}\ \ \phi \in \Hom(\o(\cL_1),\o(\cL_2))\ \ \text{and}\ \ g \in \Gal^\delta(\cL_1\oplus \cL_2).$$
\end{lemma}
\begin{proof}
Let $e$ denote the identity element in $\Galdelta(\cU)$.
From \eqref{eqn:co}, we find that, for all $x \in \omega(\cL_1)$,  
 \beq\label{eqn:co2} g_1
\zeta_\cU\big({g_1}^{-1}\big)\big(g_1^{-1} x\big)= \zeta_\cU(e)(x) -\zeta_\cU(g_1)(x)=
-\zeta_\cU(g_1)(x).\eeq
Applying repeatedly \eqref{eqn:co}, we deduce that
\begin{align*}
\zeta_\cU&\left( g_1 g_2
{g_1}^{-1}\right)(x)= g_1\left(\zeta_\cU\big(g_2{g_1}^{-1}\big)\big(g_1^{-1} x\big)\right) +
\zeta_\cU(g_1)(x) \\
 &= g_1\left( g_2\zeta_\cU\big({g_1}^{-1}\big)\big(g_2^{-1}g_1^{-1}x\big)+\zeta_\cU(g_2)\big(g_1^{-1}x\big) \right)+\zeta_\cU(g_1)(x)\\
 &= g_1 \zeta_\cU(g_2)\big(g_1^{-1}x\big) + g_1g_2g_1^{-1}\left ( g_1 \zeta_\cU \big(g_1^{-1}\big)\big(g_1^{-1} g_1g_2^{-1}g_1^{-1}x\big)\right) + \zeta_\cU(g_1)(x), \end{align*}
for all $x \in \omega(\cL_1)$. 
Since  $$g_1g_2 {g_1}^{-1},\ g_1
g_2^{-1} {g_1}^{-1}\in R_u(\Galdelta(\cU))=\Stabdelta(\cL_1 \oplus \cL_2),$$ we get that, for all $x \in \omega(\cL_1)$, 
\begin{align*}
g_1g_2g_1^{-1}&\left( g_1 \zeta_\cU \big(g_1^{-1}\big)\big(g_1^{-1} g_1g_2^{-1}g_1^{-1}x\big)\right) + \zeta_\cU(g_1)(x)\\
& = g_1 \zeta_\cU \big(g_1^{-1}\big)(g_1^{-1}x) + \zeta_\cU(g_1)(x) =0.
\end{align*}
 We conclude that, for all $x \in \omega(\cL_1)$,
$$\zeta_\cU\big( g_1 g_2 {g_1}^{-1}\big)(x)=g_1\zeta_\cU(g_2)\big(g_1^{-1} x\big).\eqno \qed$$
\end{proof}

\begin{proof}[Proof of Theorem \ref{thm:structuregaloisgroupsemidirectproduct}]
By the above, $\Gal^\delta(\cU)$ is an extension of $\Gal^\delta(\cL_1\oplus \cL_2)$ by  $R_u(\Gal^\delta(\cU))$.  
The action of $\Gal^\delta(\cL_1\oplus \cL_2)$ on $R_u(\Gal^\delta(\cU))$ is deduced from the action by conjugation of $\Gal^\delta(\cU)$
on its unipotent radical.

Combining Lemma \ref{lemma:unipcocycle} and Lemma \ref{lemma:compact}, we can identify via $\zeta_\cU$, the unipotent radical $R_u(\Gal^\delta(\cU))$
with a $\delta$-closed subgroup of $\Hom(\o(\cL_1),\o(\cL_2))$ and the action of $\Gal^\delta(\cL_1\oplus \cL_2)$  on $R_u(\Gal^\delta(\cU))$ by conjugation with the action of  $\Gal^\delta(\cL_1\oplus \cL_2)$ on $\Hom(\o(\cL_1),\o(\cL_2))$, induced by the $\Gal^\delta(\cL_1\oplus \cL_2)$-module structure on $\o(\cL_1 \oplus \cL_2)$.
\qed\end{proof}

\begin{remark}
The extension in Theorem~\ref{thm:structuregaloisgroupsemidirectproduct} does not split in general. For example,  $$G=\left\{\begin{pmatrix}a & 0 & 0\\ 0 & 1 & b\\ 0 & 0 & 1\end{pmatrix}\in\GL_3(\k)\:\Bigg| \: \delta(b)=\frac{\delta(a)}{a}\right\}$$ is a linear differential algebraic group such that the quotient map $G\to G/R_u(G)\cong\k^\times$ does not have any $\delta$-polynomial section. Indeed, otherwise, $G$ would have a projection onto $R_u(G)\cong C = \k^\delta$, which is impossible, because $G$ is strongly connected~\cite[Example 2.25]{CassSingerJordan}.
\end{remark}

\begin{remark}\label{rmk:reducetoradical}
If $K=\k(x)$ and $\p=\frac{\p}{\p x}$, the knowledge of $R=R_u(\Gal^\delta(\cU))$ allows one to compute $G=\Gal^\delta(\cU)$ algorithmically. Indeed, one can compute the normalizer $N$ of $R$ in $\GL(\omega(\cU))$. Note that $G\subset N$. By the differential version of the Chevalley theorem~\cite[Theorem~5.1]{diffreductive} (see also \cite[proof of Theorem~5.6]{Borel}), there is $\cU_0\in\{\cU\}^{\otimes,\delta}$ and a differential representation $\varrho: N\to \GL(\omega(\cU_0))$ such that $R=\ker\varrho$. The proof of this Chevalley theorem leads to a constructive procedure to find $\cU_0$ and $\varrho$. Since $\Gal^\delta(\cU_0)=\varrho(G)$ is reductive, one can compute it~\cite{MinOvSing}. We can find $G$ as $\varrho^{-1}(\Gal^\delta(\cU_0))$.
\end{remark}

In view of Remark~\ref{rmk:reducetoradical}, our aim is to compute the parameterized differential Galois group of $\cU$.  
To this purpose, we will perform a first reduction that 
will allow us to simplify our computation.

\subsection{A first reduction} \label{subsec:red}

Let $L_1,L_2 \in K[\p]$ be two completely reducible $\p$-operators. Let us 
denote the $\p$-module over $K$ corresponding
to $L_1(y)=0$ (respectively, $L_2(y)=0$)  by $\cL_1$ (respectively, by $\cL_2$).  The $\p$-module $\cU$ corresponding 
to $L_1(L_2(y))=0$ is an extension of  $\cL_1$ by $\cL_2$,
\beq \label{eq: exactseqmod} \txymatrix{
0 \ar[r] &\cL_2 \ar[r]^{i} & \mathcal{U} \ar[r]^{p} & \cL_1
\ar[r] & 0},\eeq
in the category of   $\p$-modules over $K$.  In this section, we 
recall the methods of \cite{BeSing}
to show that we can restrict ourselves to the case
 in which $L_1$ is of the form $\p - \frac{\p b}{b}$ for some $b \in K^*$.

We first describe the reduction process in terms of $\p$-modules.  Since the 
functor $\underline{\Hom}(\cL_1,-)$ is exact, \eqref{eq: exactseqmod} gives the exact sequence:

\beq \label{eq: exactseqmodhom} \txymatrix{
0 \ar[r] & \underline{\Hom}(\cL_1,\cL_2) \ar[r]  & \underline{\Hom}(\cL_1, \mathcal{U}) \ar[r]  &\underline{\Hom}(\cL_1, \cL_1)
\ar[r] & 0}\ \eeq
We pull back \eqref{eq: exactseqmodhom} by the diagonal embedding  $$d : \mathbf{1} \rightarrow \underline{\Hom}(\cL_1, \cL_1), \quad\lambda \mapsto \lambda \id_{\cL_1},$$ where $\bold{1}$ is the unit object. We obtain an exact sequence
\beq \label{eq: exactseqred} \txymatrix{
0 \ar[r] & \underline{\Hom}(\cL_1,\cL_2)  \ar[r] & \cR(\cU)  \ar[r] &\bold{1}
\ar[r] & 0},\eeq
where $\cR(\cU)$ is the $\p$-module deduced from $\cU$ by the pullback.
We call the $\p$-module $\cR(\cU)$ {\em the
reduction} of $\cU$. We recall that,  as a $K$-vector space,   $\cR(\cU)$ coincides
with the set $$\left\{(\phi,\lambda) \in   \underline{\Hom}(\cL_1, \cU) \times \bold{1}\:\big|\: p \circ \phi= \lambda\id_{\cL_1}\right\}.$$ 

\begin{remark}
An effective interpretation of this reduction process in terms of matrix differential equations immediately follows from~\cite[page~15]{BeSing}.
\end{remark}

\begin{proposition} \label{prop:red}
With notation above, we have 
\begin{enumerate}
\item The parameterized differential Galois group $\Galdelta(\underline{\Hom} (\cL_1,\cL_2))$ is a quotient 
of $\Galdelta(\cL_1 \oplus \cL_2)$ and is a reductive linear differential algebraic group;
\item  By Lemma \ref{lemma:unipcocycle}, one can identify $R_u(\Galdelta(\cU))$ (respectively, $R_u(\Galdelta(\cR(\cU)))$)
with a differential algebraic subgroup of $\Hom(\omega(\cL_1),\omega(\cL_2))$ (respectively, of $\Hom\left(\k,\Hom(\omega(\cL_1),\omega(\cL_2))\right)$).
  Then the canonical isomorphism  $$ \phi :\Hom \left( \k,\Hom(\omega(\cL_1),\omega(\cL_2)) \right) \rightarrow 
\Hom(\omega(\cL_1),\omega(\cL_2)),\ \psi \mapsto \psi(1)$$ induces an isomorphism of linear differential algebraic groups between $R_u(\Galdelta(\cR(\cU)))$  and $R_u(\Galdelta(\cU))\,$;
\item  By Lemma \ref{lemma:compact},  $\Galdelta(\cL_1 \oplus \cL_2)$ (respectively,
$\Galdelta(\underline{\Hom} (\cL_1,\cL_2))$) acts on $R_u(\Galdelta(\cU))$ (respectively, on $R_u(\Galdelta(\cR(\cU)))$).
These actions are compatible with the isomorphism $\phi$.
\end{enumerate} 
\end{proposition}
\begin{proof}

\begin{enumerate}
\item[]
\item Since $\underline{\Hom}(\cL_1,\cL_2)$ (respectively, $\cL_1 \oplus \cL_2$) 
is a subobject of $\{ \cU\}^{\otimes,\delta}$, its  parameterized differential Galois group is 
a quotient of $\Galdelta(\cU)$ by  $\Stabdelta(\underline{\Hom}(\cL_1,\cL_2))$ 
(respectively,  by $\Stabdelta(\cL_1 \oplus \cL_2) =\Stabdelta(\cL_1) \cap \Stabdelta(\cL_2)$).
It is not difficult to see that we have the inclusion  $$\Stabdelta(\cL_1 \oplus \cL_2)\subset \Stabdelta(\underline{\Hom}(\cL_1,\cL_2))$$ 
Since stabilizers of objects  in $\{\cU\}^{\otimes,\de}$ are normal in $\Galdelta(\cU)$
by Lemma~\ref{lemma:stabnormal}, we can apply \cite[Proposition 2]{Cassunip} to get that 
$$\Galdelta(\underline{\Hom}(\cL_1,\cL_2))=\Galdelta(\cU)\big/\Stabdelta(\underline{\Hom}(\cL_1,\cL_2))$$
is a quotient of $$\Galdelta(\cL_1 \oplus \cL_2)= \Galdelta(\cU)\big/\Stabdelta(\cL_1 \oplus \cL_2)$$
by $$\Stabdelta(\underline{\Hom}(\cL_1,\cL_2))\big/\Stabdelta(\cL_1 \oplus \cL_2).$$
The same reasoning  in the non-parameterized case shows that 
$\Gal(\underline{\Hom}(\cL_1,\cL_2))$ is a quotient of 
$\Gal(\cL_1 \oplus \cL_2)$.
  Since quotients of reductive algebraic groups are reductive,  \cite[Remark 2.9]{MinOvSing} allows us to conclude
  that $\Galdelta(\underline{\Hom}(\cL_1,\cL_2))$ is a reductive linear differential algebraic group.
 
 \item Since $\cR(\cU)$ is an object of $\{\cU\}^{\otimes,\delta}$, 
$\Galdelta(\cR(\cU))$ is a quotient of $\Galdelta(\cU)$, and we denote
  the canonical surjection by $\pi$.
The image of $\Stabdelta(\underline{\Hom}(\cL_1,\cL_2)) $ via $\pi$ coincides with the stabilizer of $\underline{\Hom}(\cL_1,\cL_2)$
in $\Galdelta(\cR(\cU))$ and, thus,  with $R_u(\Galdelta(\cR(\cU)))$ by Lemmas~\ref{lemma:stabunip} and~\ref{lemma:unipcocycle}.

  Let $H \subset R_u(\Galdelta(\cR(\cU)))$ be the image of $\Stabdelta(\cL_1 \oplus \cL_2)$ by $\pi$.
By \cite[Proposition~7, page~908]{cassdiffgr}, $H$ is a  differential algebraic subgroup of 
 $R_u(\Galdelta(\cR(\cU)))$.  Since 
 $\Stabdelta(\cL_1 \oplus \cL_2)$
 is normal in $\Galdelta(\cU)$ and $\pi$ is surjective,  $H$ is normal in $R_u(\Galdelta(\cR(\cU)))$,  and we can consider the quotient map $$p: R_u(\Galdelta(\cR(\cU))) \rightarrow R_u(\Galdelta(\cR(\cU)))\big/H\,.$$
 Since quotients of unipotent linear differential algebraic groups are unipotent by \cite[Theorem~1]{Cassunip},  the linear differential algebraic group $R_u(\Galdelta(\cR(\cU)))/H$ is unipotent.
Note that
\begin{equation}\label{eq:quotbysum}
R_u\big(\Galdelta(\cR(\cU))\big)\big/H=\pi\big( \Stabdelta(\underline{\Hom}(\cL_1,\cL_2))\big)\big/ \pi\big( \Stabdelta(\cL_1 \oplus \cL_2)\big)
\end{equation}
The surjective morphism $\pi$ is induced via $\delta$-Tannakian equivalence
by the inclusion of $\delta$-Tannakian categories $\{\cR(\cU)\}^{\otimes,\de} \subset \{\cU\}^{\otimes,\de}$. This inclusion restricts
to the inclusion of  the usual Tannakian categories $\{\cR(\cU)\}^{\otimes} \subset \{\cU\}^{\otimes}$, which shows, taking the Zariski closure, that $\pi$
extends to a surjective morphism of algebraic groups $\overline{\pi}: \Gal(\cU) \rightarrow
\Gal(\cR(\cU))$. One can show that the quotient $$\overline{\pi}(\Stab(\underline{\Hom}(\cL_1,\cL_2)))\big/ \overline{\pi}( \Stab(\cL_1 \oplus \cL_2))$$
coincides with the Zariski closure of $R_u(\Galdelta(\cR(\cU)))/H$.

Let  $K_{\cL_1 \oplus \cL_2}$ (respectively, $K_{\underline{\Hom}(\cL_1,\cL_2)}$)  denote the usual PV extension of $\cL_1 \oplus \cL_2$ (respectively, of 
$\underline{\Hom}(\cL_1,\cL_2)$) over $K$. Let  $K_{\cU}$ (respectively, $K_{R(\cU)}$)  denote the usual PV extension of $\cU$ (respectively, of 
$\cR(\cU))$) over $K$.  We have the following tower of $\p$-field extensions:
$$
\txymatrix{
& K_\cU \ar@{-}[ld] \ar@{-}[rd] & \\
K_{R(\cU)} \ar@{-}[rd] & & K_{\cL_1 \oplus \cL_2}\ar@{-}[ld] \\
& K_{\underline{\Hom}(\cL_1,\cL_2)}\ar@{-}[d] & \\
 & K&}\
$$
 We see that $$\Gal\left(K_{\cL_1 \oplus\cL_2}\big/K_{\underline{\Hom}(\cL_1,\cL_2)}\right)=\Stab(\underline{\Hom}(\cL_1,\cL_2))\big/
 \Stab(\cL_1 \oplus \cL_2)\,.$$ Since $K_{\underline{\Hom}(\cL_1,\cL_2)}$ is a PV extension of $K$, the group
 $\Gal\left(K_{\cL_1 \oplus \cL_2}\big/K_{\underline{\Hom}(\cL_1,\cL_2)}\right)$ is normal in $\Gal\left(K_{\cL_1 \oplus \cL_2}/K\right)$ by 
 the PV correspondence. Therefore, $\Gal\left(K_{\cL_1 \oplus \cL_2}\big/K_{\underline{\Hom}(\cL_1,\cL_2)}\right) $ is a reductive algebraic group. Since 
 \begin{align*}
 \overline{\pi}: \Stab(\underline{\Hom}(\cL_1,\cL_2))&\big/
 \Stab(\cL_1 \oplus \cL_2))\\
  &\rightarrow \overline{\pi}\big( \Stab(\underline{\Hom}(\cL_1,\cL_2))\big)\big/ \overline{\pi}\big( \Stab(\cL_1 \oplus \cL_2)\big)
 \end{align*}
 is a quotient map, we deduce from the above identifications  that the Zariski closure of $R_u(\Galdelta(\cR(\cU)))/H$  is a reductive algebraic group. We 
 conclude by   \cite[Remark 2.9]{MinOvSing} that $R_u(\Galdelta(\cR(\cU)))/H$ is reductive. 
On the other hand, since $R_u(\Galdelta(\cR(\cU)))/H$ is both unipotent and reductive, it must be equal to $\{e\}$, and we have 
\begin{equation}\label{eq:equalimages}\pi\big(\Stabdelta(\cL_1 \oplus \cL_2)\big)=\pi\big( \Stabdelta(\underline{\Hom}(\cL_1,\cL_2))\big)=R_u(\Galdelta(\cR(\cU)))\,.
\end{equation}
  We recall the notation of Lemma \ref{lemma:unipcocycle}. We denote by 
$s$ a $\k$-linear section of the exact sequence of finite-dimensional representations of $\Galdelta(\cU)$:
$$
\txymatrix{
0 \ar[r] &\omega(\cL_2) \ar[r]^{\omega(i)} &
\omega(\mathcal{U}) \ar[r]^{\omega(p)} & \omega(\cL_1) \ar@{.>}@/^/[l]^{s}
\ar[r] & 0}\,.$$ 
Then, we identify $R_u(\Galdelta(\cU))=\Stabdelta(\cL_1 \oplus \cL_2)$
with the image of $\Stabdelta(\cL_1 \oplus \cL_2)$
by
$$\zeta_\cU : R_u(\Galdelta(\cU)) \rightarrow \Hom_\k(\o(\cL_1),\o(\cL_2))\,,\quad
g \mapsto  \big( x \mapsto  g s(
g^{-1}x ) - s(x)\big)\,.$$ 
Since $\o$ is compatible with $\underline\Hom$, 
the map    
$$r : \k \rightarrow \o(\cR(\cU)),\quad \lambda \mapsto (\lambda s, \lambda),$$
is a $\k$-linear section of $t$
$$\txymatrix{
0 \ar[r] & \Hom(\o(\cL_1),\o(\cL_2)) \ar[r] &\o(\cR(\cU)) \ar[r]^{\quad t} & \k \ar@{.>}@/^/[l]^{\quad r}
\ar[r] & 0}.$$
We apply again   Lemma~\ref{lemma:unipcocycle} to  identify $R_u(\Galdelta(\cR(\cU)))=\pi\big(\Stabdelta(\cL_1 \oplus \cL_2)\big)$
with its image 
via \begin{align*}\zeta_{\cR(\cU)} : \Galdelta(\cR(\cU)) &\rightarrow  \Hom(\k,\Hom_\k(\o(\cL_1),\o(\cL_2)))\\
g &\mapsto \big( \lambda \mapsto  g r(\lambda)
g^{-1} - r(\lambda)\big)\,.\end{align*}
Identifying   $\Hom(\k,\Hom(\o(\cL_1),\o(\cL_2)))$ with $\Hom(\o(\cL_1),\o(\cL_2))$ via $\phi$, we find
that \beq \label{eq:compacocycl}\zeta_\cU = \phi \circ \zeta_{\cR(\cU)} \circ \pi.\eeq  We have
\begin{align*}R_u\big(\Galdelta(\cU)\big)&=\zeta_\cU\big(\Stabdelta(\cL_1 \oplus \cL_2)\big)\\
&=\zeta_{\cR(\cU)} \circ \pi\big(\Stabdelta(\cL_1 \oplus \cL_2)\big) =R_u\big(\Galdelta(\cR(\cU))\big),
\end{align*}
where we have used Remark~\ref{rem:indep}.
\item The compatibility of the actions comes from Lemma \ref{lemma:compact},  \eqref{eq:compacocycl}, and~\eqref{eq:equalimages}.\qed
\end{enumerate}
\end{proof}

We combine  Proposition \ref{prop:red} and Theorem \ref{thm:structuregaloisgroupsemidirectproduct} in the following Theorem. 
 
 \begin{theorem}\label{thm:redparamradunip}
If $\cL_1,\cL_2$ are completely reducible $\p$-modules over $K$ 
 and if $\cU$ is a $\p$-module extension of $\cL_1$ by $\cL_2$,
 then
 \begin{enumerate}
 \item 
 $\Galdelta(\cU)$ is an extension of $\Galdelta(\cL_1 \oplus  \cL_2)$
 by a $\delta$-subgroup $W$ of 
 $ \o(\underline{\Hom}(\cL_1,\cL_2)) $. 
 \item $W=R_u(\Galdelta(\cR(\cU)))$, where $\cR(\cU)$ is an extension of  $\mathbf{1}$ by
 the completely reducible $\p$-module $\underline{\Hom}(\cL_1,\cL_2)$, and  the action of $\Gal^\delta(\cL_1 \oplus  \cL_2)$
 on $W$ is given by composing the quotient map of $\Galdelta(\cL_1\oplus\cL_2)$ on $\Gal^\delta(\underline{\Hom}(\cL_1,\cL_2))$
 with the
 action of $\Gal^\delta(\underline{\Hom}(\cL_1,\cL_2))$ on $\o(\underline{\Hom}(\cL_1,\cL_2))$.
 \end{enumerate}
 \end{theorem}

\subsection[The unipotent radical of the parameterized differential Galois group of an extension]{The unipotent radical of the parameterized differential Galois group of an extension of $\bold{1}$ by a completely reducible $\p$-module $\cL$}\label{sec:computation}
Let $\cL$ be a completely reducible $\p$-module over $K$ and $\cU$ be an extension of $\bold{1}$ by $\cL$.
In this section, we study $R_u(\Galdelta(\cU))$.

In terms of $\p$-operators, the situation corresponds to the following.
Let $L \in K[\p]$ be a completely reducible $\p$-operator and  $\cL$ be the associated $\p$-module.
An extension $\cU$ of $\bold{1}$ by $\cL$  corresponds to an inhomogeneous differential equation of the form $L(y)=b$ for some $b \in K^*$.
The main result of \cite{BeSing}  is to show 
that $R_u(\Gal(\cU))=\o(\cL_0)$, where $\cL_0$ is the largest $\p$-module
of $\cL$ such that 
\begin{enumerate}
\item $L=L_1L_0\,$;
\item $L_1(y)=b$ has a solution in $K$.
\end{enumerate}
From Lemma \ref{lemma:unipcocycle}, we know that  $R_u(\Galdelta(\cU))$ can be identified with 
a differential algebraic subgroup $W$ of $\omega(\cL_0)$, stable under the natural action of $\Galdelta(\cL)$
on $\omega(\cL)$.  In \cite{HardouinSemCongres},  the result of \cite{BeSing} was rephrased in Tannakian 
terms and it was proved that  $\cL_0$  is the smallest subobject of $\mathcal{L}$ such that the pushout of the extension $\cU$
by the quotient map $ \pi : \mathcal{L} \rightarrow \mathcal{L}/\cL_0$ is a trivial (split) extension. Such a characterization 
 no longer holds  in general in the parameterized setting.
Indeed,  the classification  of  differential algebraic subgroups of vector groups shows that $W$ coincides with the zero set 
of a finite system of linear  homogeneous differential equations with coefficients in $\k$.  Therefore,  we have two possibilities:
\begin{itemize}
\item either  $W$ is given by linear homogeneous polynomials and it is a finite-dimensional vector space over $\k$, that is, $W$
is an algebraic subgroup of $\o(\cL_0)\,$; 
\item or   $W$ is given by linear homogeneous $\delta$-polynomials of order greater than $0$, and $W$ is a
vector space over $C=\k^\delta$.
\end{itemize}
 In the first case, 
we deduce from the $\delta$-Tannakian equivalence for the category $\{\cL\}^{\otimes,\delta}$  that $W=\o(\widetilde{\cL_0})$ for a submodule $\widetilde{\cL_0}$ of $\cL$ if and only if it is an algebraic subgroup of $\o(\cL_0)$. In this situation, we show that $\widetilde{\cL_0}$ is the smallest $\p$-submodule
of $\cL$ such that the parameterized differential Galois group of the pushout of the extension $\cU$
by the quotient map $ \pi : \mathcal{L} \rightarrow \mathcal{L}/\widetilde{\cL_0}$ is reductive (see Theorem~\ref{thm:purelynonconstantuniprad}). This last condition can 
be tested by an algorithm contained in \cite{MinOvSing}.

If $W$ is not given by linear homogeneous $\delta$-polynomials of order $0$, then $W$ is not of the form $\omega(\widetilde{\cL})$ for any $\widetilde\cL$. Moreover, the order of the defining equations of $W$ can be as high as required even for second order differential equations:
\begin{example}For $n \ge 0$, let $$z(x,t,n) = \sum_{j=0}^n t^j\ln(x+j)\,;\quad a(x,t,n) = \frac{\partial z(x,t,n)}{\partial x} =  \sum_{j=0}^n \frac{t^j}{x+j} \in \k(x)\,,$$
where $\k$ is a differentially closed field with respect to $\partial/\partial t$ containing $\mathbb{Q}(t)$. Then the function $z(x,t,n)$ satisfies the following second order differential equation in $y(x,t)$ over $\k(x)\,$:
$$
\frac{\partial\left(\frac{\partial y(x,t)}{\partial x}\big/a(x,t,n)\right)}{\partial x}=0\quad\iff\quad \frac{\partial^2y(x,t)}{\partial x^2}-\frac{\frac{\partial a(x,t,n)}{\partial x}}{a(x,t,n)}\frac{\partial y(x,t)}{\partial x} =0.
$$
Since $\ln(x),\ldots,\ln(x+n)$ are algebraically independent over $\k(x)$ by \cite{OstrovskiActa,DHWrelations}, and $\frac{\partial^{n+1} z(x,t,n)}{\partial t^{n+1}}=0\,$,  and 
$$
\k(x)(\ln(x),\ldots,\ln(x+n)) = \k(x)\left(\frac{\partial^j(z(x,t,n))}{\partial t^j}\:\Big|\: j \ge 0\right)\,,
$$
we have 
$$
\Galdelta = \left\{\begin{pmatrix}1& a\\0&1\end{pmatrix}\:\Big|\: \frac{\partial^{n+1}a}{\partial t^{n+1}}=0\right\}\,.
$$ 
\end{example}

In Section~\ref{subsec:decompconstantnonconst}, we give a  decomposition of  $\cL$  into ``constant and purely non-constant'' parts, which 
allows us to distinguish between the two cases for the unipotent radical $W$ described above.  In  Section~\ref{sec:purelynonconstantcase}, we  treat the ``purely non-constant case''. In Section~\ref{sec:mergeconstantnonconstant}, we give a general algorithm
to compute $R_u(\Galdelta(\cU))$ under the assumption that $\cL$ has no non-zero trivial $\p$-submodules in the sense  of Definition~\ref{defn:trivial objects}.

\subsubsection{Decomposition of the completely reducible $\p$-module $\cL$}\label{subsec:decompconstantnonconst}

 The following lemma  gives a  decomposition of  a completely reducible
 $\p$-module into a direct sum of $\p$-modules, a ``constant'' one and a ``purely non-constant'' one. 

\begin{lemma}\label{lemma:decompconst}
Let $\cL $ be a completely reducible $\p$-module and $\rho :\Galdelta(\cL) \rightarrow \Gl(\o(\cL))$ be the 
representation of the parameterized differential Galois group of $\cL$ on $\o(\cL)$. Then  there exist $\p$-submodules $\cL_c$ and $\cL_{nc}$ of 
$\cL$ such that
\begin{itemize}
\item $\cL=\cL_{c} \oplus \cL_{nc}\,$;
\item the representation of $\Galdelta(\cL)$ on $\cL_c$ is conjugate to constants in $\Gl(\o(\cL_c))$, that is, 
any differential system associated to $\cL_c$ is isomonodromic by Proposition \ref{prop:compintconjconts};
\item $\cL_c$ is maximal for the properties above, that is, there is no non-zero $\p$-submodule $\mathcal N$  of $\cL_{nc}$
such that the representation of $\Galdelta(\cL)$ on $\mathcal N$ is conjugate to constants in $\Gl(\o(\mathcal{N}))$.
\end{itemize} 
\end{lemma}
\begin{proof}
 Let $\cL_1, \dots, \cL_r$ be irreducible $\p$-submodules 
such that $\cL=\cL_1 \oplus \ldots \oplus \cL_r$. We have $$\Gl(\o(\cL)) = \prod_{i=1}^r \Gl(\o(\cL_i))\,.$$  Let $S$ be the set of indices $i$  in $\{1,\dots,r\}$
such that  the representation of $\Galdelta(\cL)$ on $\omega(\cL_i)$ is conjugate to constants in $\Gl(\o(\cL_i))$. Setting   $$\cL_c =\bigoplus_{i \in S} \cL_i\quad  
\text{and}\quad \cL_{nc}= \bigoplus_{i \notin S} \cL_i$$ allows to conclude the proof.
\qed\end{proof}

\begin{remark}
The above construction is effective.  Let $\cL$ be a completely reducible $\p$-module over $K=\C(z)$ with $\p(z)=1$ and $\p(\C)=0$. There are many algorithms that  compute a factorization of $\cL$ into a direct sum of  irreducible $\p$-submodules: see, for instance, \cite{vanHoeijfactor,singtestred}. Thus, we can find a linear differential system associated to $\cL$ of the form 
$$\p (Y)=\begin{pmatrix}A_1 & 0&  \dots & 0 \\
0 &A_2 & \dots & 0  \\
  \vdots & \ddots & \ddots &\vdots \\
  0 & \dots & 0 & A_r \end{pmatrix} Y$$ with $A_i \in K^{n_i \times n_i}$ for all $i=1,\dots,r$ and such that $\p(Y)=A_i Y$ is an irreducible differential system. For all $i=1,\dots,r$, let  $\cL_i$  be a $\p$-module associated to $\p(Y)=A_i Y$. Let $S$ be the set of indices $i$ such that  there exists a matrix $B_i \in K^{n_i \times n_i}$ such that  $$ \delta(A_i)-\p(B_i)=B_iA_i -A_iB_i\,.$$  Since there are algorithms to find rational solutions of linear differential systems (see \cite{barkatou}), the construction of the set $S$ is also effective.
  We can set $$\cL_c =\bigoplus_{i \in S} \cL_i\quad  
\text{and}\quad \cL_{nc}= \bigoplus_{i \notin S} \cL_i\,.$$
    \end{remark}
This decomposition motivates the following definition.

\begin{definition}
A $\p$-module $\cL$ over $K$ is said to be constant if the representation of $\Galdelta(\cL)$ on $\o(\cL)$ is conjugate to constants in $\GL(\o(\cL))$.
On the contrary, the $\p$-module $\cL$ is said to be \textit{purely non-constant} if  there is no non-zero $\p$-submodule $\mathcal N$  of $\cL$
such that the representation of $\Galdelta(\cL)$ on $\o(\mathcal N)$ is conjugate to constants in $\Gl(\o(\mathcal{N}))$.
\end{definition}

\begin{remark}We say that  a  $G$-module $V$ is {\em purely non-constant} if, for every non-zero $G$-submodule $W$ of $V$, the induced representation $\rho: G \rightarrow \GL(W)$ is non-constant. By the Tannakian equivalence,
a $\p$-module $\cL$ is purely non-constant if and only if the $\Galdelta(\cL)$-module
$\o(\cL)$ is purely non-constant.
\end{remark}

Recall that $\cU$ is a $\p$-module extension of $\bold{1}$ by $\cL$.
We consider the  pushout of 
$$\txymatrix{
0 \ar[r] &\cL \ar[r] & \cU \ar[r] & \bold{1}
\ar[r] & 0}\ 
$$
by the projection of $\cL$ on $\cL_c$ (respectively, on 
$\cL_{nc}$). We find two exact sequences of $\p$-modules:
 \beq\label{eq:exactseqconst2}  \txymatrix{
0 \ar[r] &\cL_c \ar[r] & \cU_c \ar[r] & \bold{1}
\ar[r] & 0},\eeq 
and 
 \beq \label{exactseqnonconst2}  \txymatrix{
0 \ar[r] &\cL_{nc} \ar[r] & \cU_{nc} \ar[r] & \bold{1}
\ar[r] & 0}.\eeq 
We deduce from Lemma \ref{lemma:unipcocycle} that
\begin{itemize}
\item $R_u(\Galdelta(\cU))$ is a differential algebraic subgroup  of $\o(\cL)\,$;
\item $R_u(\Galdelta(\cU_c))$ is a differential algebraic subgroup  of $\o(\cL_c)\,$;
\item $R_u(\Galdelta(\cU_{nc}))$ is a differential algebraic subgroup  of $\o(\cL_{nc})$.
\end{itemize}

The quotient $\Galdelta(\cU_c)\big/ R_u(\Galdelta(\cU_c))$ is $\Galdelta(\cL_c)$, which is, by construction, conjugate to constants.
We can use \cite{MinOvSingunip} to compute $R_u(\Galdelta(\cU_c))$. Section \ref{sec:purelynonconstantcase} shows how to compute the unipotent radical of  the parameterized differential Galois group of  an extension of $\bold{1}$ by a purely
 non constant completely reducible module.  Finally, Section~\ref{sec:mergeconstantnonconstant} shows how to combine Section~\ref{sec:purelynonconstantcase}
 with \cite{MinOvSingunip} to deduce $R_u(\Galdelta(\cU))$ from  the computation of $R_u(\Galdelta(\cU_c))$ and  $R_u(\Galdelta(\cU_{nc}))\,$.

\subsubsection{The purely non-constant case}\label{sec:purelynonconstantcase}

The aim of this section is to prove the following theorem.

\begin{theorem}\label{thm:purelynonconstantuniprad}
Let $\cL$ be a purely non-constant completely reducible $\p$-module over $K$. Let $\cU$
be a {$\p$-module extension} of $\bold{1}$ by $ \cL$. Then,  $R_u(\Galdelta(\cU))=\o(\widetilde{\cL_0})$, where $\widetilde{\cL_0}$ is the smallest $\p$-submodule of $\cL$
 such that $\Galdelta(\cU/\widetilde{\cL_0})$ is reductive.
\end{theorem}

By Theorem \ref{thm:redparamradunip},  $R_u(\Galdelta(\cU))$ is a $\delta$-closed subgroup  of 
$\o(\cL)$, which is stable under the action of $\Galdelta(\cL)$.  We  show that any such subgroup  is a $\k$-vector subspace.
In this attempt, we  first treat the cases in which $\Galdelta(\cL)$ is a torus or  $\SL_2$. We   conclude with the general situation and the proof of Theorem~\ref{thm:purelynonconstantuniprad}.

The algorithm contained in \cite{MinOvSing} allows one to test whether the unipotent radical of a linear algebraic group is trivial. This algorithm relies on bounds
 on the order of the defining equations of the parameterized differential Galois group.
Combined with Theorem \ref{thm:purelynonconstantuniprad}, we find a complete algorithm to compute $R_u(\Galdelta(\cU))$.

Theorem \ref{thm:purelynonconstantuniprad} implies among other things that  $R_u(\Galdelta(\cU))$ is an
algebraic subgroup of $R_u(\Gal(\cU))$. Despite the fact that $\Galdelta(\cU)$ (respectively, $\Galdelta(\cL)$) is Zariski dense
in $\Gal(\cU)$ (respectively, $\Gal(\cL)$), it might happen that $R_u(\Galdelta(\cU))$ is contained in a proper Zariski closed subgroup of $R_u(\Gal(\cU))$ as it is shown in the following example.
\begin{example}\label{exa:paramradunipalgradunip}
Let $V=\Span_\k\{x^2,xy,y^2,x'y-xy'\}\subset \k\{x,y\}$, and let us consider the following representation $\rho:\PSL_2 \to \GL(V)$ (cf. \cite[Example~3.7]{MinOvRepSL2}):
\begin{equation}\label{eq:repSL2}
\begin{pmatrix}
a&b\\
c& d
\end{pmatrix}\
\text{mod}\ \left\{\begin{pmatrix}
1&0\\
0& 1
\end{pmatrix},\begin{pmatrix}
-1&0\\
0& -1
\end{pmatrix}\right\}
\mapsto
\begin{pmatrix}
a^2 & ab & b^2 &a'b-ab'\\
2ac & ad+bc& 2bd & 2(bc'-ad')\\
c^2 & cd & d^2 &c'd-cd'\\
0 &0&0& 1
\end{pmatrix}\,.
\end{equation}
Note that  $\overline{\rho(\PSL_2)} = \Ga^3  \rtimes \PSL_2 $, and we have: $R_u(\PSL_2)=\{e\}$ whereas $R_u(\Ga^3  \rtimes \PSL_2)=\Ga^3$.
By \cite[Theorem~1.1 and Lemma~2.2]{Singerinv}, we can construct a $\p$-module $\cU$ such that $\Galdelta(\cU)=\PSL_2$, and $\rho$ is the representation of $\Galdelta(\cU)$ on $\omega(\cU)$ (so that $\Gal(\cU) = \Ga^3  \rtimes \PSL_2$). We can also construct a $\p$-module $\cL$ such that $\cU$ is an extension of $\bold{1}$ by $\cL$ in the given representation. 
\end{example}

For a subset $B$ of a $\k$-vector space $V$, we denote $\k B$ the smallest $\k$-subspace of $V$ that contains $B$. Note that $\k B$ consists of all finite linear combinations of elements of $B$ with coefficients in~$\k$.

\begin{proposition}\label{prop:unipradvectorgroup}
Let $G$ be a reductive linear differential algebraic group and $V$  a purely non-constant completely reducible
$G$-module. Then every $G$-invariant $\delta$-subgroup $A\subset V$ is a submodule.
\end{proposition}
\begin{proof}
We need only to show that $A$ is $\k$-invariant. Let us assume that $G$ is connected. The general case will follow by Propositions~\ref{prop:cr} and~\ref{prop:connectedcompconjconst}, which imply that $V$ is completely reducible and purely non-constant as a $G^\circ$-module.

Let us prove that $A$ is $\k$-invariant by induction on $\dim V$. Let $B$ be minimal among the non-zero $G$-invariant $\delta$-subgroups of $V$ that are contained in $A$, which exists by the Ritt--Noetherianity of the Kolchin topology. In what follows, we shall prove that $\k B = B$. Assuming this, by the semisimplicity of $V$, let $W \subset V$ be a $G$-invariant $\k$-subspace such that $V = B\oplus W$. Then $A = B\oplus (W\cap A)$, and $\k (W\cap A)= W\cap A$ by the inductive hypothesis. Therefore, $\k A = A$.

Let us show that there exists $x\in \k\setminus{C}$ such that $xB=B$. 
Since $V$ is purely non-constant, $V'=\k B$ is purely non-constant, and so it contains a simple non-constant submodule $U$. By Corollary~\ref{cor:diag}, there exists a $\delta$-torus $T\subset G$ such that $U$ semisimple and non-constant as a $T$-module. By the construction of $T$ (see the proof of Corollary~\ref{cor:diag}) and Proposition~\ref{prop:tensorprod}, every simple $G$-module is semisimple as a $T$-module. Therefore, $V$ and $V'$ are semisimple as $T$-modules. Hence, $\overline{T}$ is an algebraic torus, and there is a direct sum of weight spaces
\begin{equation}\label{eq:weightsum}
V'=\bigoplus_{\chi} V'_\chi
\end{equation}
over all algebraic characters $\chi:\overline{T}\to\k^\times$. By definition,
$$
V'_\chi=\big\{v\in V'\:|\:\ t(v)=\chi(t)v\ \ \text{for all } t\in \overline{T}\big\}.
$$
Note that $V'_\chi$, viewed as $C$-linear spaces, are weight spaces with respect to $\overline{T}(C)=T_C$. Since any character $\chi$ (being defined by monomials) is uniquely determined by its restriction to $\overline{T}(C)$, the direct sum~\eqref{eq:weightsum} is also the weight space decomposition of the $C$-space $V'$ with respect to the action of $T_C$.  Since$T_C\subset T\subset G$ and the $\delta$-subgroup $B\subset V'$ is $G$-invariant, $B$ is also $T_C$-invariant. Moreover, $B$ is a $C$-vector space~\cite[Proposition~11]{cassdiffgr}. Therefore, we have the weight decomposition of the $C$-space with respect to the action of $T_C$:
$$
B=\bigoplus_{\chi} B_\chi,\qquad\text{where}\qquad B_\chi=\left(B\cap V'_\chi\right).
$$
Since $V'=\k B$, $V'_\chi=\k B_\chi$. In particular, $B_\chi$ is non-zero if $V'_\chi$ is. By the definition of $T$, there is a character $\chi$ of $\overline{T}$ such that $\chi(T)\not\subset C$ and $V'_\chi\neq\{0\}$. Therefore, there exist $b\in B_\chi$, $b\neq 0$, and $t\in T$ such that $t$ acts on $b$ by multiplication by a non-constant element $x$. We fix such an~$x$. Due to the $G$-invariance of $xB$, we obtain that $B\cap xB$ is a $G$-invariant non-trivial $\delta$-subgroup of $B$. Since $B$ is minimal, $xB=B$.

On the one hand, the set $S = \{a \in \k\:|\: aB\subset B\}$ is a $C$-subalgebra of $\k$. 
On the other hand,
 $$S = \bigcap_{b\in B}\varphi_b^{-1}(B),\quad \varphi_b:\k\to V,\quad t\mapsto tb\,,$$
is a $\delta$-subgroup of $\k$. Therefore, by~\cite[Theorem II.6.3, page~97]{Kol}, $S=C$ or $\k$. Since $x\in S$, $S=\k$.
\qed\end{proof}

\begin{proof}[Proof of Theorem \ref{thm:purelynonconstantuniprad}]
By Theorem \ref{thm:redparamradunip},  $R_u(\Galdelta(\cU))$ is a $\delta$-closed subgroup $W$ of 
$\o(\cL)$ which is stable under the action of $\Galdelta(\cL)$. Proposition \ref{prop:unipradvectorgroup} shows that $W$ is 
a $\k$-vector space and thereby a $\Galdelta(\cL)$-module. By $\delta$-Tannakian equivalence for the category $\{\cL\}^{\otimes,\delta}$, we obtain that  $W$ is of the form $\o(\cW)$ for some
$\p$-submodule $\cW\subset\cL\subset\cU$. Thus, it remains to prove that $\cW$ is the smallest $\p$-submodule $\widetilde{\cL_0}$ of $\cL$ such that the parameterized differential Galois group of $\cU/\widetilde{\cL_0}$ is reductive. 

Let us show that the set  $\bold{V}$ of subobjects  $\mathcal{W}$ of $\mathcal{L}$ such that
 $R_u(\Galdelta(\mathcal{U}/\mathcal{W}))=\{1\}$  admits a smallest subobject with respect to the inclusion. It is enough to prove that, if $\mathcal{V}_1$ and $\mathcal{V}_2$ belong to $\bold{V}$, their intersection $\mathcal{W}$ lies in $\bold{V}$. Denote by $G$, $G_1$, and $G_2$ the parameterized differential Galois groups of $\mathcal{U}/\mathcal{W}$, $\mathcal{U}/\mathcal{V}_1$, and $\mathcal{U}/\mathcal{V}_2$, respectively. The quotient maps $\mathcal{U}/\mathcal{W}\to\mathcal{U}/\mathcal{V}_i$ give rise to homomorphisms $\varphi_i:G\to G_i$, $i=1,2$. Since $G_i$ are reductive, $R_u(G)\subset \ker\varphi_i$. Therefore, it suffices to show that $\ker\varphi_1\cap \ker\varphi_2=\{1\}$. For each $g\in G$, the condition $g\in \ker\varphi_i$ means that $g(u)-u\in \o(\cV_i)$ for all $u\in \o(\cU)$. Therefore, every element of $\ker\varphi_1\cap \ker\varphi_2$ acts trivially on $\o(\cU)/\o(\cW)$. 

As in the notation of Lemma \ref{lemma:unipcocycle}, let $s$ be a $\k$-linear section of the last arrow of the following exact sequence
$$
0\rightarrow \o(\cL) \rightarrow \o(\cU) \rightarrow \k \rightarrow 0
$$
and let $\zeta_\cU$ be its associated cocycle. By  Lemma \ref{lemma:unipcocycle} and Proposition \ref{prop:unipradvectorgroup},
the cocycle $\zeta_\cU$ identifies $R_u(\Galdelta(\cU))$ with  a $\k$-vector subgroup $W=\o(\cW)$ of $\o(\cL)$ for some
$\p$-submodule $\cW\subset\cU$. To conclude the proof, we have to show that $W=\o(\widetilde\cL_0)$.

It follows from the definition of $\zeta$ that the  diagram 
\begin{equation}\label{eq:diag1}
\txymatrix{
\Galdelta(\cU) \ar[d]^\varrho \ar[r]^{\zeta_{\cU}} & \o(\cL) \ar[d]^{\beta}\\
\Galdelta(\cU/\cW) \ar[r]^{\zeta_{\cU/\cW}} & \o(\cL/\cW)}\
\end{equation}
where the vertical arrows are induced by the quotient maps, is commutative. By the definition of $\cW$ and exactness of~$\o$, the composition $\beta\zeta_{\cU}$ vanishes on~$R_u(\Galdelta(\cU))$. Since $\o(\cU/\cW)$ is a faithful $\Galdelta(\cU/\cW)$-module and $\o(\cL/\cW)$
has no non-zero trivial $\Galdelta(\cL/\cW)$-submodule by assumption, and therefore no  non-zero trivial  $\Galdelta(\cU/\cW)$-submodules by assumption,
Propositions~\ref{prop:aux} and~\ref{rem:aux} below show that $$R_u(\Galdelta(\cU/\cW))=\rho(R_u(\Galdelta(\cU)))\,.$$ Since $\zeta$ is one-to-one on the unipotent radical, we conclude that the linear differential algebraic group $\Galdelta(\cU/\cW)$ is reductive. Therefore, $\cW\supset\widetilde{\cL_0}$. If we replace $\cW$ with a $\p$-submodule $\cV\subset\cU$ in the above diagram such that $\Galdelta(\cU/\cV)$ is reductive, we obtain that $$\o(\cV)\supset\zeta_{\cU}(R_u(\Galdelta(\cU)))=W\,.$$ Thus, $\o(\widetilde\cL_0)\supset W$. \qed\end{proof}

Recall that unipotent linear differential algebraic groups are connected. (Otherwise they would have unipotent finite quotients, which is impossible.) Therefore, for every linear differential algebraic group $G$, we have  $R_u(G)=R_u(G^\circ)=R_u(G)^\circ$.

\begin{proposition}\label{prop:aux}
Let $\varrho: G\to H$ be  a surjective homomorphism of linear differential algebraic groups. Assume  that, for every proper subgroup   $N\subset R_u(H)$ that is normal in $H$, the group $ R_u(H/N)$ is not central in $(H/N)^\circ=H^\circ /N$. Then $\varrho(R_u(G))=R_u(H)$. 
\end{proposition}
\begin{proof}
Let $N=\varrho(R_u(G))\subset R_u(H)$. By the surjectivity of $\varrho$, the group $N$ is normal in $H$.
Consider the epimorphism of quotients
$$
\nu:G/R_u(G)\to H/N
$$
induced by $\varrho$. The linear differential algebraic group $ \nu^{-1}(R_u(H/N))^\circ$ is normal in the reductive linear differential algebraic group $(G/R_u(G))^\circ$. Therefore, it is reductive itself. By Theorem~\ref{thm:decompalmostdirectprodreductive}, 
$ \nu^{-1}(R_u(H/N))^\circ$    is an almost direct product of a $\delta$-closed subgroup $Z$ of a central torus $T \subset (G/R_u(G))^\circ$ and of 
quasi-simple linear differential algebraic groups  $H_i$. Since the subgroups $H_i$ coincide with their commutator groups, they cannot have unipotent images unless $\nu(H_i)=\{e\}$. We conclude that $\nu(Z)=R_u(H/N)$. Since $Z$ is central in $(G/R_u(G))^\circ$ and $\nu$ is surjective, the group
$\nu(Z)$ is central in $(H/N)^\circ$. It follows from the  assumption that $N=R_u(H)$. 
\qed\end{proof}

\begin{proposition}\label{rem:aux}
The assumption on $H$ in Proposition~\ref{prop:aux} is satisfied if there exists a short exact sequence 
$$
0\to V\to U\to \bold{1}\to 0
$$
of $H^\circ$-modules, where $U$ is a faithful $H^\circ$-module  and $V$ is a  $H^\circ$-semisimple module with no non-zero trivial $H^\circ$-submodule.
\end{proposition}

\begin{remark}
Note that if the $H^\circ$-module $V$ has no trivial $H^\circ$-submodules,
then $V$ has no  no zero  $C$-vector space fixed by the action of $H^\circ$.
Indeed, let $f$ be a nonzero element of a $C$-vector space fixed by $H^\circ$, then the $\k$-vector space spanned by $f$ is fixed by $H^\circ$.
\end{remark}
\begin{proof}
It suffices to prove the statement for connected $H$. Let $N\subset R_u(H)$ be a $\delta$-subgroup that is normal in $H$ and such that $R_u(H/N)$ is central in $H/N$. Since we have a commutative diagram
$$
\txymatrix{
H \ar[r] & H/N \\
R_u(H) \ar@{^{(}->}[u] \ar[r] & R_u(H/N), \ar@{^{(}->}[u]}\ $$
 the latter implies that, for all $g\in R_u(H)$, one has $hgh^{-1}\in gN$. Let $u\in U$ be an element whose image in $\bold{1}$ is non-zero. Moreover, $R_u(H)$
 acts trivially on $V$ because $V$ is $H$-semi-simple. Thus, the map 
$$
\zeta: R_u(H)\to V, \qquad g\mapsto gu-u
$$ 
is an $H$-equivariant one-to-one homomorphism of linear differential algebraic groups (see proofs of Lemmas~\ref{lemma:unipcocycle} and~\ref{lemma:compact}), that is, for all $h \in H$ and $g \in R_u(H)$, we have
$$
hgu-hu = hgh^{-1}u-u.
$$
The $\delta$-subgroups $\zeta(R_u(H))$ and $\zeta(N)$ of $V$ are thus stable under the action of $H$. Note that $\zeta(R_u(H))$ and $\zeta(N)$
are $C$-vector spaces since, as $\delta$-subgroup of $V$, they are zero sets linear homogeneous differential equations over $\k$.

 Let $n \in N$ be such that $hgh^{-1}=gn$ and $n' \in N$ be such that $gng^{-1}=n'$. Then
\begin{align*}
h(gu-u) &= hgu-hu = gnu-u =n'gu-u+n'u-n'u\\
&= n'(gu-u) +n'u-u=gu-u+n'u-u,
\end{align*}
since $gu-u \in V$ and $R_u(H)$ acts trivially on $V$.
Therefore, $H$ acts trivially on $\zeta(R_u(H))/\zeta(N)$.  Since $\zeta(R_u(H))$ is $H$-semisimple as $H$-module over $C$, the $H$-module  $$\zeta(R_u(H))/\zeta(N) \subset \zeta(R_u(H)) \subset V$$
is a $C$-vector space fixed  by the action of $H$. This contradicts the assumption on $V$. It follows that $R_u(H)=N$.\qed\end{proof}

 \subsubsection{A general algorithm}\label{sec:mergeconstantnonconstant}

Will will  explain a general algorithm to compute the unipotent radical of a $\p$-module extension $\cU$
of $\bold{1}$ by a completely reducible $\p$-module $\cL$. We recall that $\cL$ can be decomposed as the direct sum of a constant $\p$-module $\cL_c$
and a purely non-constant $\p$-module $\cL_{nc}$. Considering the pushouts of the extension $\cU$ with respect to the decomposition of $\cL$, 
 we find the following two exact sequences of $\p$-modules:
 $$   \txymatrix{
0 \ar[r] &\cL_c \ar[r] & \cU_c \ar[r] & \bold{1}
\ar[r] & 0},$$
and 
$$\txymatrix{
0 \ar[r] &\cL_{nc} \ar[r] & \cU_{nc} \ar[r] & \bold{1}
\ar[r] & 0}.$$
We assume that $K=\k(x)$ so that we can use  the algorithm contained in \cite{MinOvSingunip} to compute $R_u(\Galdelta(\cU_c))$ and the algorithm of Section~\ref{sec:purelynonconstantcase} to compute $R_u(\Galdelta(\cU_{nc}))$. The quotient map $\cU \rightarrow \cU/\cU_c=\cU_{nc}$ induces an epimorphism $\alpha: \Galdelta(\cU) \rightarrow \Galdelta(\cU_{nc})$.
Similarly, we find an epimorphism $\beta:  \Galdelta(\cU) \rightarrow \Galdelta(\cU_{c})$. The following theorem allows us to compare $R_u(\Galdelta(\cU))$  with 
the groups computed above.

\begin{theorem}\label{thm:decomp}
Let $K=\k(x)$,   $\cL, \cU,\cU_c,\cU_{nc}$ be as above. Assume that $\cL$ has no non-zero  trivial $\p$-submodule.
 Then the map $$\alpha\times\beta:R_u(\Galdelta(\cU))\to R_u(\Galdelta(\cU_{nc}))\times R_u(\Galdelta(\cU_c))$$ is an isomorphism of linear differential algebraic groups.
\end{theorem}
\begin{proof}
We will use  the notion of \emph{differential type} $\tau(G)$ of a linear differential algebraic group $G$ (see \cite[Section~2.1]{CassSingerJordan} and \cite[Definition~2.2]{MinOvSingunip}). 
Recall that, in the ordinary case, $\tau$ can only take the values $-1$, $0$, or $1$. We will also use the following result: 
\begin{lemma}[{\cite[Equation (1), page~195]{CassSingerJordan}}]\label{lem:inequtype}
Let   $G$  be a linear differential algebraic group and $H$ be  a normal differential algebraic subgroup of $G$. Then $\tau(G) = \max\{\tau(H),\tau(G/H)\}$ .\end{lemma}

Let us consider the commutative diagram:
\begin{equation}\label{eq:diag2}
\txymatrix{
R_u((\Galdelta(\cU_c)) \ar@{^{(}->}[d]& R_u((\Galdelta(\cU)) \ar@{^{(}->}[d] \ar[l]_{\beta} \ar[r]^{\alpha} & R_u((\Galdelta(\cU_{nc})) \ar@{^{(}->}[d]\\
\o(\cU_c)  & \o(\cU)=\o(\cU_c)\oplus \o(\cU_{nc}) \ar[l]\ar[r] & \o(\cU_{nc})
}\
\end{equation}
Here, the vertical arrows correspond to embedding (that is, a one-to-one homomorphism) via the associated cocycles (see~\eqref{eq:diag1}). The horizontal arrows of the lower row correspond to natural projections. Note that $R_u((\Galdelta(\cU_c))$, $R_u((\Galdelta(\cU))$, and $R_u((\Galdelta(\cU_{nc}))$ are all
abelian groups (see Theorem \ref{thm:structuregaloisgroupsemidirectproduct}). 
It follows from~\eqref{eq:diag2} that $\alpha\times\beta$ is an embedding. Then, by \cite[Corollary~2.4]{CassSingerJordan} and Lemma \ref{lem:inequtype}, 
\begin{align*}\tau\big(R_u(\Galdelta(\cU)\big) &\le \tau \big(R_u(\Galdelta(\cU_c)) \times R_u(\Galdelta(\cU_{nc}))\big) 
\\&= \max\big\{\tau\big(R_u(\Galdelta(\cU_c))\big),\tau\big(R_u(\Galdelta(\cU_{nc}))\big)\big\}\,.\end{align*} Since $\alpha$ and $\beta$ are surjective, we find that $$\tau\big(R_u(\Galdelta(\cU)\big) = \max\big\{\tau\big(R_u(\Galdelta(\cU_c))\big),\tau\big(R_u(\Galdelta(\cU_{nc}))\big)\big\}\,.$$  If  $R_u(\Galdelta(\cU_{nc}))\ne\{e\}$, it is  isomorphic to  a non-trivial  vector group over $\k$ and  its differential type is $1$ (see \cite[Example 2.9]{CassSingerJordan}).   Moreover, since the unipotent radicals considered above are $\delta$-closed subgroups of vector groups,  they are either algebraic groups and their differential type is $1$, or finite-dimensional $C$-vector spaces of differential type $0$. If $R_u(\Galdelta(\cU_{nc})=\{e\}$, we have nothing to prove. Thus, we  assume that $R_u(\Galdelta(\cU_{nc})\ne\{e\}$ and that   its differential type is $1$. By the discussion above, we can also assume that $$\tau(R_u(\Galdelta(\cU)))=1.$$
Since $\cL$ has no non-zero  trivial  $\p$-submodule, the same holds for $\cL_c$ and $\cL_{nc}$. 
By Propositions~\ref{prop:aux} and~\ref{rem:aux}, $\alpha$ and $\beta$ are surjective. Let $R_0\subset R_u(\Galdelta(\cU))$ stand for the strong identity component of $R_u(\Galdelta(\cU))$ (\cite[Definition~2.6]{CassSingerJordan}). 
  Since $R_u(\Galdelta(\cU_{nc}))$ is algebraic by Theorem \ref{thm:purelynonconstantuniprad}, it is strongly connected by \cite[Lemma~2.8 and Example~2.9]{CassSingerJordan}.  
 We have  $$\alpha(R_0)=R_u(\Galdelta(\cU_{nc}))$$ (Indeed, otherwise $\alpha(R_0) \subsetneq R_u(\Galdelta(\cU_{nc}))$. By definition of the strong
identity component, we find that  $$\tau\big(R_u(\Galdelta(\cU))/R_0\big) <1.$$ However,  
 $$\tau(R_u(\Galdelta(\cU_{nc}))/ \alpha(R_0))=1,$$ because $R_u(\Galdelta(\cU_{nc}))$ is strongly connected. Therefore, we have 
a  surjective map $$R_u(\Galdelta(\cU) )/R_0\to R_u(G_{nc})/\alpha(R_0)$$ from a linear differential algebraic group of differential type  smaller than $1$  onto a linear differential algebraic group of differential type $1$, which is impossible. 
 Therefore, the group product map $$R_0\times\ker\alpha\to R_u(\Galdelta(\cU)),\quad (r_0, x)\mapsto r_0x$$ is onto. To finish the proof, it suffices to show that $$\beta(\ker\alpha)=R_u(\Galdelta(\cU_c)).$$  
  If $\beta(R_0)\ne \{e\}$, it is strongly connected and $$\tau(\beta(R_0)) = \tau(R_0)=1.$$ Since $\tau\big(R_u(\Galdelta(\cU_{nc}))\big)=0$  (see \cite[Theorem~2.13]{MinOvSingunip}), we have $\beta(R_0)=\{e\}$ (by Lemma~\ref{lem:inequtype}). 
Thus, $$\beta(\ker \alpha)=R_u(\Galdelta(\cU_{nc})).\eqno
\qed$$
\end{proof}


\section{Criteria of hypertranscendance}\label{sec:criteria}
We start with a new result in the representation theory of quasi-simple and reductive linear differential algebraic groups, which we  further use for a hypertranscendence criterion.
\subsection{Extensions of the trivial representation}\label{sec:repsplit}

Let $(\k, \delta)$ be a $\delta$-closed field such that $\Char\k =0$ and let $C$ be its field of $\delta$-constants. Let  {$G\subset\GL_n(\k)$} be a connected linear differential algebraic group over $\k$.  We recall the definition of the Lie algebra of $G$, following \cite[Chapter 3]{cassdiffgr}. 
\begin{definition}
A $\k$-linear derivation $D$ of the field of fractions $\k\langle G\rangle$ of the $\delta$-coordinate ring $\k\{G\}$ of $G$  is called a \textit{differential derivation} if $D\circ \delta=\delta\circ D$.
\end{definition}

 In particular, every differential derivation is determined by its values on the matrix entries that differentially generate $\k\{G\}$ and, therefore, can be represented by an $n\times n$ matrix. The group $G$ acts by right translations on the set of differential derivations of $\k\langle G\rangle$. 
 
 \begin{definition}
 The set $\Lie G$ of invariant differential derivations, denoted also by $\ga$, is called the \emph{Lie algebra} of $G$. 
 \end{definition}
 
 This is a $C$-Lie subalgebra of the Lie algebra $\mathfrak{gl}_n(\k)=\Lie\Gl_n(\k)$ of all $n\times n$ matrices. Moreover, $\ga$ is also a $\delta$-subgroup of the additive group of $\mathfrak{gl}_n(\k)$. 
Every $\delta$-homomorphism of linear differential algebraic groups gives rise (by taking the differential) to a $C$-homomorphism of their Lie algebras.  We refer to~\cite[Chapter 3]{cassdiffgr} for the details. 

\begin{definition}
A \emph{$\ga$-module} (respectively, $C$-$\ga$-module) is a finite-dimensional $\k$-vector space (respectivelty, $C$-vector space, possibly infinite-dimensional) $V$ together with a $C$-Lie algebra homomorphism 
$\nu:\ga \rightarrow \mathfrak{gl}(V)$, where $\mathfrak{gl}(V)$ denotes the Lie algebra of $\k$-linear endomorphisms of $V$. 
\end{definition}

Every $G$-module $V$ is also a $\ga$-module, where $\nu=d\rho: \ga \rightarrow \mathfrak{gl}(V)$ is   the differential (see \cite[pages~928-929]{cassdiffgr}) of the homomorphism $\rho: G\to\Gl(V)$. (Formally, to agree with the above definitions, we assume that a basis of $V$ is chosen, hence we can identify $\Gl(V)$ and $\mathfrak{gl}(V)$ with $\Gl_n(\k)$ and $\mathfrak{gl}_n(\k)$, respectively.) The definitions of simple, semisimple, and other types of $\ga$-modules that we use here are analogues to those for $G$-modules.

It follows from \cite[Proposition 20]{cassdiffgr} that, if $G\subset\Gl_n(\k)$ is given by polynomial equations, then $\Lie G$ coincides with the Lie algebra of the group $G$ considered as an algebraic group. Moreover, for an arbitrary linear differential algebraic group $G\subset\Gl_n(\k)$, the Lie algebra $\Lie\overline{G}$ of its Zariski closure $\overline{G}$ coincides with the $\k$-span of $\Lie G$ in $\mathfrak{gl}_n(\k)$. Recall that, in the case of $G = \overline{G}$, $\Lie G$ is a $G$-module, which is called \emph{adjoint}, where the action of $G$ is induced from its action on $\mathfrak{gl}_n(\k)$ by conjugation. The differential of the corresponding homomorphism $\Ad: G\to\Gl(\ga)$ gives the $\k$-Lie algebra map $\mathrm{ad}: \ga\to\mathfrak{gl}(\ga)$ defining the structure of the $\ga$-module on $\ga$, also called \emph{adjoint}.  One has $(\mathrm{ad} x)(y)=[x,y]$ for all $x,y\in\ga$.

For any group, Lie algebra, or ring $R$, we denote the set of $R$-module homomorphisms by $\Hom_R(V,W)$.

For a $C$-Lie algebra $\ga$, let $\ga_\k=\k\otimes_C\ga$ denote the $\k$-Lie algebra with the bracket  determined by
$$
[x\otimes\xi,y\otimes\eta]=xy\otimes[\xi,\eta] \qquad\forall x,y\in\k,\quad\xi,\eta\in\ga.
$$
We have the inclusion $$\ga\simeq C\otimes\ga\subset\k\otimes\ga=\ga_\k.$$
If $\ga \subset \ha$ are Lie algebras, then we also consider $\ha$ as a $\ga$-module under the adjoint action.

\begin{lemma}\label{lem:LieAlg}
Let $H\subset\Gl_n(C)$ be a reductive algebraic group and $\ha=\Lie H\subset\mathfrak{gl}_n(C)$. Let  $\ga\subset\ha_k$ be a $C$-Lie subalgebra containing $\ha$ and
\begin{equation}\label{eq:exseq0}
0\to V\to W\to\bold{1}\to 0
\end{equation}
an exact sequence of $\ga$-modules (over $\k$). 
If
\begin{enumerate}
\item[(1)] sequence~\eqref{eq:exseq0} splits as a sequence of $\ha$-modules and
\item[(2)] $\Hom_{\ha_\k}(\ha_\k, V)=0$ (in other words, $V$ does not contain quotients of the adjoint representation of $\ha_\k$),
\end{enumerate}
then sequence~\eqref{eq:exseq0} splits.
\end{lemma}
\begin{proof}
If one chooses a basis $\{e_1,\ldots,e_{n-1},e_n\}$ of $W$ such that $V=\Span\{e_1,\ldots,e_{n-1}\}$, then the matrix $\varrho(\xi)\in\mathfrak{gl}(W)$ corresponding to $\xi\in\ga$ can be written in the form
$$
\begin{pmatrix}
\alpha(\xi) & \varphi(\xi)\\ 0 & 0
\end{pmatrix},
$$
where $\alpha:\ga\to\mathfrak{gl}(V)$ determines the $\ga$-module structure on $V$ and $\varphi:\ga\to V$ is a $C$-linear map. The fact that $\varrho$ defines a homomorphism of Lie algebras is the 
following condition on $\varphi$:
\begin{equation}\label{eq:one}
\varphi\left([\xi,\eta]\right)=\alpha(\xi)\varphi(\eta)-\alpha(\eta)\varphi(\xi)\qquad\forall\xi,\eta\in\ga.
\end{equation}
Choosing another vector for $e_n$, one obtains another $C$-linear map $\varphi':\ga\to V$, which is called equivalent to $\varphi$. Sequence \eqref{eq:exseq0} splits if and only if $\varphi$ is equivalent to $0$.

Let us choose $e_n$ in such a way that  
\begin{equation}\label{eq:two}
\varphi(\xi)=0\qquad\forall\xi\in\ha,
\end{equation}
which is possible due to assumption~(1).
It follows from~\eqref{eq:one} and~\eqref{eq:two} that
\begin{equation}\label{eq:xieta}
\varphi\left([\xi,\eta]\right)=\alpha(\xi)\varphi(\eta)\qquad\forall\xi\in\ha,\quad\eta\in\ga.
\end{equation}
Since $H$ is reductive, by \cite[page~97, Theorem]{Waterhouse:IntrotoAffineGroupSchemes} and \cite[Chapter~2]{SpringerInv}, there exist simple $\ha$-submodules $\ha_1,\ldots,\ha_m$ in $\ha$ such that $\ha=\bigoplus\limits_{i=1}^m\ha_i$. Let $B \subset \k$ be a $C$-basis of $\k$ as a $C$-vector space. For each $a \in \k$ and $i$, $1\le i \le m$, $a\otimes \ha_i$ is a simple $C$-$\ha$-submodule of $\ha_\k$ and
\begin{equation}\label{eq:bhisum}
\ha_\k =\bigoplus_{1\le i\le m\atop{b \in B}} b\otimes\ha_i.
\end{equation}
For every $C$-$\ha$-submodule $I \subset \ha_\k$, let $I'$ be a maximal sum of the simple components in decomposition~\eqref{eq:bhisum} with $I' \cap I = \{0\}$. Such an $\ha$-submodule $I'$ exists by Zorn's lemma.  We will show that 
\begin{equation}\label{eq:hkissum}
 \ha_\k = I\oplus I'.
 \end{equation} Let  $S = b\otimes \ha_i$ for some $b \in B$ and $1\le i \le m$. If $S \cap \big(I \oplus I'\big) = \{0\}$, then  $I \cap \big(S\oplus I'\big) = \{0\}$. Indeed, if $v \in I$ and $v =v_1+v_2$, where $v_1 \in S$ and $v_2 \in I'$, then  $v_2 = v-v_1 \in S \cap \big(I \oplus I'\big)$, and so $v=v_1 \in I\cap S = \{0\}$. By the maximality of $I'$, $S \subset I'$, which contradicts $S \cap \big(I \oplus I'\big) = \{0\}$. Therefore, 
\begin{equation}\label{eq:Sisin}
S \cap \big(I \oplus I'\big) \ne \{0\}.
\end{equation}  
Since $S$ is a simple $\ha$-module,~\eqref{eq:Sisin} implies that $S \subset I\oplus I'$. 
Thus,~\eqref{eq:hkissum} holds and therefore  $\ha_\k$ is a semisimple $\ha$-module. (cf. \cite[\S 4.1]{Bourbaki8}).

The $C$-$\ha$-module $\ga$ is semisimple. Indeed, every $\ha$-invariant subspace $J\subset\ga$ has a complementary invariant subspace $J'$ in $\ha_k$, since $\ha_k$ is semisimple. Therefore, $$\ga=J\oplus \left(J'\cap\ga\right).$$ Thus, to prove that $\varphi$ is the zero map, it suffices to show that $\varphi(J)=\{0\}$ for every simple $C$-$\ha$-submodule $J\subset\ga$. Since such $J$ is isomorphic to $\ha_i$ for some $i$, $1\le i \le m$, we have the $\ha$-equivariant $C$-linear map 
$$
\mu:\ha\stackrel{\pi}{\to}\ha_i\simeq J\subset\ga\stackrel{\varphi}{\to} V,
$$
where $\pi$ is the projection with respect to an $\ha$-invariant decomposition $\ha=\ha_i\oplus\ha_i'$, and the $\ha$-equivariance of $\varphi$ is implied by~\eqref{eq:xieta} . Since $\mu$ extends to the $\k$-linear $\ha_k$-equivariant map $\ha_k\to V$, assumption~(2) yields that $\mu$ is the zero map. Therefore, $\varphi(J)=\{0\}$. 
\qed\end{proof}

\begin{lemma}\label{lem:45} Let $G$ be a connected linear differential algebraic group and $\ga$ be its Lie algebra. Any $G$-module $W$ is completely reducible if and only if it is completely reducible as a $\ga$-module.
\end{lemma}
\begin{proof} Let $G_W$ denote the image of $G$ in $\GL(W)$. The $G$-module $W$ is completely reducible if and only if it is completely reducible as a $G_W$-module. The latter is equivalent to $W$ being completely reducible as a $\overline{G_W}$-module. Since $\Char \k= 0$, this is equivalent to the semisimplicity of $W$ viewed as the $\Lie\overline{G_W}$-module (see \cite[page~97, Theorem]{Waterhouse:IntrotoAffineGroupSchemes}). Since $\Lie\overline{G_W}$ is the $\k$-span of $\Lie G_W\subset\mathfrak{gl}(W)$, $W$ is completely reducible as a $\Lie\overline{G_W}$ if and only if it is completely reducible as a $\Lie G_W$-module. Since, by \cite[Proposition~22]{cassdiffgr}, $\Lie G_W$ is an image of $\ga$ in $\gl(W)$, $W$ is completely reducible as a $\ga$-module if and only if $W$ is completely reducible as a $\Lie G_W$-module.
\qed\end{proof}

\begin{theorem}\label{thm:adjoint}
Let $G$ be a connected linear differential algebraic group over $\k$ and
\begin{equation}\label{eq:exseq}
0\to V\to W\to\bold{1}\to 0
\end{equation}
an exact sequence of $G$-modules, where $V$ is faithful and semisimple. Let $\overline{G}$ denote the Zariski closure of $G$ in $\Gl(V)$. 
If  $V$, viewed as a $\overline{G}$-module, does not contain non-zero submodules isomorphic to a quotient of the adjoint module for $\overline{G}$, that is, if
$$
\Hom_{\overline{G}}(\Lie\overline{G}, V)=0,
$$
then sequence~\eqref{eq:exseq} splits.
\end{theorem} 
\begin{proof}
By Lemma~\ref{lem:45}, it is sufficient to show that $W$ is completely reducible as a $\ga$-module.
Since $G$ admits a faithful completely reducible representation (given by $V$), it is reductive. Therefore,
by \cite[Lemma 4.5]{diffreductive}, there is a $\delta$-isomorphism $\nu: \wtilde{H}\to G$, where $\wtilde{H}\subset\Gl_r(\k)$ is a $\delta$-group such that its $\delta$-subgroup $H_C=\wtilde{H}\cap\Gl_r(C)$ is Zariski dense (the Zariski topology on $\wtilde{H}$ is induced from $\Gl_r(\k)$). 

Let $H=\nu(H_C)$ and $\ha=\Lie H$. We will show that $\ha$ and $\ga$ satisfy the hypotheses of Lemma~\ref{lem:LieAlg}, which would thus yield the proof (in particular, we will identify $\ga$ with a subalgebra of $\ha_\k$).
The differential algebraic group $H\simeq H_C$ is reductive. Indeed, if its unipotent radical were non-trivial, $\overline{R_u(H_C)}\cap \wtilde{H}$ would be a non-trivial normal unipotent differential algebraic subgroup of $\wtilde{H}$, which is impossible due to the reductivity of $G\simeq\wtilde{H}$.

Let us show that $\nu$ extends to an algebraic isomorphism $\overline\nu:\overline{H_C}\to\overline{G}$ of the Zariski closures. By \cite[Theorem 3.3]{diffreductive}, this would follow if the $G$-module $V$ is completely reducible and $\overline{H_C}$ is reductive. It only remains to prove the latter.  Since $H_C$ is reductive, $C^r$ is a completely reducible $H_C$-module. Therefore, $\k^r$ is completely reducible as an $\overline{H_C}$-module. Thus, $\overline{H_C}$ is reductive. 

The differential $d\overline\nu$ defines an isomorphism between $\k$-Lie algebras $\Lie\overline{H_C}$ and $\Lie\overline{G}$. Since $\Lie{H_C}\subset\mathfrak{gl}_r(C)$ and any $C$-basis of $\mathfrak{gl}_r(C)$ is also a $\k$-basis of $\mathfrak{gl}_r(\k)$, we obtain that any $C$-basis of $\Lie{H_C}$ is $\k$-linearly independent. Since $\Lie\overline{H_C}$ is the $\k$-span of $\Lie{H_C}$, we can therefore write $$\Lie\overline{H_C}=\k\otimes_C\Lie{H_C}.$$ Applying $d\overline\nu$, this implies that $$\Lie\overline{G}=\k\otimes_C\ha=\ha_\k.$$ Therefore, we have $$\ha\subset\ga\subset\ha_\k.$$
Since every $\delta$-representation of $H_C$ is polynomial and $\overline{H_C}$ is reductive, every $\delta$-representation of $H_C$ is completely reducible. Therefore, $W$ is completely reducible as an $H$-module (and $\ha$-module), and so sequence~\eqref{eq:exseq} splits as a sequence of $\ha$-modules. 
Finally, using \cite[page~97, Theorem]{Waterhouse:IntrotoAffineGroupSchemes} and $\Lie\overline{G} = \ga_\k$, we conclude that $$\Hom_{\ga_\k}(\ga_\k, V)=\Hom_{\Lie\overline{G}}\big(\Lie\overline{G}, V\big)=\Hom_{\overline{G}}\big(\Lie\overline{G}, V\big)=0.\eqno\qed$$ 
\end{proof}

\subsection{A practical criterion of hypertranscendance}\label{sec:criterion}
Let $\Delta=\{\p,\delta\}$ be a set of two derivations.
Let $K$ be a $\Delta$-field such that $K^\p=\k$ (recall that $\k$ is $\delta$-closed).  From the results of the previous sections, we obtain the 
following criterion for  the hypertranscendence of the solutions of $L(y)=b$, for irreducible $L \in K[\p]$.
\begin{theorem}\label{cor:criteria}	
Let $L \in K[\p]$
be an irreducible $\p$-operator such that  $\Gal(L)$  is a quasi-simple linear algebraic group. Denote $n= \ord L$ and $m = \dim\Gal(L)$.  Suppose that $m \ne n$. Let  $b \in K^*$ and $F$ a 
 $\Delta$-field extension of $K$ such that $F^\p=\k$ and $F$ contains $z$, a solution of $L(y)=b$, and $u_1,\dots,u_n$, $K$-linearly independent solutions of $L(y)=0$. Then
 \begin{itemize}
\item the functions $v_1,\ldots,v_m, z,\dots,\p^{n-1} z$ and all their derivatives
with respect to $\delta$ are algebraically independent over $K$, where $\{v_1,\ldots,v_m\} \subset \{u_1,\dots,\p^{n-1}u_1,\dots,u_n,\dots,\p^{n-1}u_{n}\}$ is a maximal algebraically independent over $K$ subset
\end{itemize}
if and only if
\begin{itemize}
\item the linear differential system
$
\p(B)-\delta(A_L)=A_LB-BA_L$, where $A_L$ denotes the companion matrix of $L$, has no solutions $B \in K^{n \times n}$ and
\item the linear differential equation $L(y)=b$ has no solutions in $K$.
\end{itemize}
\end{theorem}
\begin{example} If $L \in K[\p]$ and $\Gal(L) = \SL_n$, where $n =\ord L \ge 2$, then $L$ is irreducible and $\dim L \ne \dim\Gal(L) = n^2-1$. In this situation, in Theorem~\ref{cor:criteria}, we can take $$\{v_1,\ldots,v_m\} =  \{u_1,\dots,\p^{n-1}u_1,\dots,u_{n-1},\dots,\p^{n-1}u_{n-1},u_n,\dots,\p^{n-2}u_{n}\}\,.$$
\end{example}
\begin{proof}[Proof of Theorem~\ref{cor:criteria}]
Let $\cL$ (respectively, $\cU$) be the $\p$-module associated to $L$ (respectively, to $(\p -\p(b)/b)L$). 
Since the $\Delta$-field $K_\cU$ generated by $u_1,\ldots,u_n, z$ in $F$ is a PPV extension for $\cU$ over $K$,
the  differential transcendence degree of $K_\cU$ over $K$ equals the differential dimension of $\Galdelta(\cU)$.
Since
$\cL$ corresponds to the differential system $\p Y=A_L Y$, Proposition \ref{prop:compintconjconts} together with Theorem~\ref{thm:decompalmostdirectprodreductive}\eqref{part3} imply that the first hypothesis is equivalent to $\Galdelta(\cL)=\Gal(\cL)$.

Since $L$ is irreducible, there is no non-zero trivial $\p$-submodule $\cN$ of $\cL$ such that  the representation of $\Galdelta(\cL)$
on $\o(\cN)$ is conjugate to constants, that is, $\cL$ is purely non-constant. By Theorem \ref{thm:purelynonconstantuniprad}, $R_u(\Galdelta(\cU))=\o(\widetilde{\cL_0})$, where $\widetilde{\cL_0}$ is the smallest $\p$-submodule of $\cL$
 such that $\Galdelta(\cU/\widetilde{\cL_0})$ is reductive. Since $\cL$ is irreducible, either $\widetilde{\cL_0}$ is zero
 or $\widetilde{\cL_0}=\cL$.  
The module $\widetilde{\cL_0}$ is   zero
if and only if $R_u(\Galdelta(\cU))=\{e\}$. Moreover, $R_u(\Galdelta(\cU))=\{e\}$ if and only if $\omega(\cU)$ is a $\Galdelta(\cL)$-module. 
Since $\dim_\k\o(\cL)=n$, the $\Galdelta(\cL)$-module $\o(\cL)$
is not adjoint. Since $\Gal(L)$ is a quasi-simple linear algebraic group, $\Lie(\Gal(L))$ is simple (see \cite[Section~14.2]{Humphline}), and therefore its adjoint representation is irreducible. This implies that $$\Hom_{\Gal(L)}(\Lie(\Gal(L)),\o(\cL))=0.$$ Therefore, by the above and Theorem~\ref{thm:adjoint}, we find that  $\widetilde{\cL_0}$ is  zero if and only if the sequence of $\Galdelta(\cL)$-modules 
\begin{equation*}\label{exactsequstranscriteria}
0\to \o(\cL) \to \o(\cU)\to\k\to 0
\end{equation*}
splits, which, by \cite[Theorem~3.5]{cassisinger}, is equivalent to the existence of a solution in $K$ of the equation $L(y)=b$, in contradiction with the second hypothesis.
Therefore, we find that the second hypothesis is equivalent to $ R_u(\Gal(\cU))=(\k^n, +)$, that is, the vector group $\mathbf{G}_a^n$ and $\Galdelta(\cU)=\mathbf{G}_a^n \rtimes \Gal(\cL)$. 
The latter is equivalent to $v_1,\ldots,v_m, z,\dots,\p^{n-1} z$ being a differential transcendence basis of $K_\cU$ over~$K$.
\qed\end{proof}

\begin{remark}
The condition in the statement of Theorem~\ref{cor:criteria} to have no solutions $B \in K^{n\times n}$ is equivalent to the fact that 
$\Galdelta(\cL)$ is not conjugate 
to constants. For $K$ a computable field, this condition can be tested through various algorithms that find rational solutions  (see, for instance,~\cite{barkatou}). However,  one can sometimes easily prove the non-integrability of the
system by taking a close look at the topological generators of the parameterized differential Galois group such as the monodromy or the Stokes matrices.
This is the strategy employed in Lemma \ref{lemma:Besselparam}.
\end{remark}

\subsection{Application to the Lommel equation}\label{sec:examples}

We apply Theorem~\ref{cor:criteria} to the differential Lommel equation, which is a  non-homogeneous Bessel equation
\begin{equation}\label{eq:lommel}
 \frac{d^2y}{dx^2}+ \frac{1}{x} \frac{dy}{dx} + \left(1 -\frac{\alpha^2}{x^2}\right) y=  x ^{\mu -1},
\end{equation}
depending on two  parameters, $\alpha, \mu \in \C$. 

We will study the differential dependence of the solutions of \eqref{eq:lommel} with respect to the parameter $\alpha$.  To this purpose, we consider $\alpha$ as a new variable, transcendental over $\mathbb{C}$, and suppose that   $\mu \in \Z$. We endow the field $\C(\alpha,x)$ with the derivations $\delta=\frac{\partial }{\partial \alpha}$ and $\partial= \frac{ \partial}{\partial x}$, $\Delta = \{\delta,\partial\}$. 
Let $\k$ be a $\delta$-closure of $\C(\alpha)$.  We extend $\partial$ to $\k$ as the zero derivation. We extend $\Delta$ to $K = \k(x)$, the field of rational functions in $x$ with coefficients in $\k$, so that $\C(\alpha,x)$ is a $\Delta$-subfield of $K$.
Indeed, let $\mathcal{A}=\k\otimes_{\C(\alpha)} \C(\alpha,x)$, which is a $\Delta$-algebra over $\C(\alpha,x)$, and $\mathcal{A}^\partial=\k$. Since $\C(\alpha,x)^\partial=\C(\alpha)$, the multiplication homomorphism $\varphi: \mathcal{A} \to K$, is injective (see \cite[Corollary~1, page~87]{Kol}). Therefore, there is an extension of $\Delta$ onto $K$ making $\varphi$ a $\Delta$-homorphism so that $\C(\alpha,x)\subset K$ is a $\Delta$-field extension via $\varphi$.

Let $\cL$ be a $\partial$-module over $K$ associated to the Bessel differential  equation 
\begin{equation}\label{eq:Besseldeq}
L(y)= \frac{d^2y}{dx^2}+ \frac{1}{x} \frac{dy}{dx} + \left(1 -\frac{\alpha^2}{x^2}\right) y=0
\end{equation} and  let $\cU$ be a $\partial$-module over $K$ associated to the Lommel differential equation. We have:
\begin{equation}\label{eq:Lommel}
0 \rightarrow \cL \rightarrow  \cU \rightarrow \bold{1} \rightarrow 0\,.
\end{equation}

\begin{lemma}\label{lemma:Besselparam}
The parameterized differential Galois group of $\cL$ over $K$ is $\Sl_2$.

\end{lemma}
\begin{proof}
The differential Galois group of $\cL$ over $K$ is known to be $\Sl_2$ (see ~\cite{Kol68}).  By \cite{Cassimpl}, we know that either $\Gal^\delta(\cM)= \Sl_2$
or $\Gal^\delta(\cL)$ is conjugate to constants in $\Sl_2$. Suppose that  we are in the second situation, that is, there exists $P \in \Sl_2$ such that  $$P\Gal^\delta(\cL)P^{-1} \subset \{ M \in \Sl_2 \:|\: \delta(M)=0\}.$$
The coefficients of \eqref{eq:Besseldeq} lie  in $\mathbb{C}(\alpha,x)$. Moreover,   for a fixed value
of $\alpha$ in $\C$,  the point zero is a parameterized regular singular point of  \eqref{eq:Besseldeq} (see \cite[Definition~2.3]{MitschiSinger:MonodromyGroupsOfParameterizedLinearDifferentialEquationsWithRegularSingularities}).
If we fix a fundamental solution $Z_0$ of \eqref{eq:Besseldeq} and follow \cite[page~922]{MitschiSinger:MonodromyGroupsOfParameterizedLinearDifferentialEquationsWithRegularSingularities}, we are able to  compute
 the parameterized  monodromy  matrices of~\eqref{eq:Besseldeq} around zero.  For a suitable choice of $Z_0$, we find  the following parameterized  monodromy  matrix,
$$
M_0 = \begin{pmatrix}
\zeta & 0 \\
0 & \overline{\zeta}
,\end{pmatrix}  
$$ where $\zeta= e^{2 i\pi \alpha}$ and $\overline{\zeta} =e^{-2 i\pi \alpha}$ (see \cite[page~35]{Morales}).
By \cite[Theorem 3.5]{MitschiSinger:MonodromyGroupsOfParameterizedLinearDifferentialEquationsWithRegularSingularities},  $M_0$ belongs  to some conjugate of  $\Gal^\delta (\cL)$. This means that there exists $Q \in \Gl_2$ such that 
 $\delta(QM_0Q^{-1})=0$. Since conjugate matrices have the same spectrum and the spectrum of $M_0$ is not $\delta$-constant, we find
 a contradiction.\qed\end{proof}
Let $J_\alpha(x)$ be the Bessel function of the first kind and let $Y_\alpha(x)$ be the Bessel function of the second kind. A solution of the Lommel differential equation is the Lommel 
 function $s_{\mu,\alpha}(x)$, which is defined as follows
$$
s_{\mu,\alpha}(x) = \frac{1}{2} \pi  \left[ Y_\alpha (x) \int_0^x x^\mu J_\alpha (x)\, dx - J_\alpha (x) \int_0^x x^\mu Y_\alpha (x)\, dx\right].$$
\begin{proposition}\label{prop:grouplommel} 
The functions, $J_\alpha(x),Y_\alpha(x), \frac{d}{dx}(Y_\alpha)(x), s_{\mu,\alpha}(x)$ and $\frac{d}{dx}s_{\mu,\alpha}(x)$ 
and all their derivatives of all order with respect to $\frac{\partial}{\partial \alpha}$ are algebraically independent over $\C(\alpha,x)$. Moreover,
the parameterized differential Galois group of $\cU$ is isomorphic to a semi-direct product $\Ga^2 \rtimes \Sl_2$. 
\end{proposition}
\begin{proof}
 
Since $\Gal^\delta(\cL)=\Sl_2$, we just need to prove that $L(y)=x^{\mu -1}$ has no solution $g$ in $K$ in order to apply Theorem~\ref{cor:criteria} to the Lommel differential equation. Thus, suppose on the contrary  that $L(y)=x^{\mu -1}$ has a rational solution
 $g \in \k(x)$. Using  partial-fraction decomposition, one can show that 
the only possible pole  of $g$ is zero. If we write $$g=\sum_{j=m}^n a_j x^j,\quad m,n\in\Z,\ m\leq n,\ a_j\in\k,\ a_ma_n\neq 0,$$ then the highest and lowest order terms of $L(g)\in\k[x,1/x]$ are $$a_nx^n\neq 0\quad \text{and}\quad (m^2-\alpha^2)a_mx^{m-2}\neq 0,$$ respectively. Since different powers of $x$ are linearly independent over $\k$ and $n\neq m-2$, $L(g)-x^{\mu-1}$ contains at least one non-zero term. Contradiction.
\qed\end{proof}

\bibliographystyle{spmpsci}
\bibliography{bibdata}

\end{document}